\documentclass[11pt,leqno]{amsart}
\usepackage{amsfonts}
\usepackage{color,xcolor}
\usepackage{graphicx}
\usepackage{manfnt}

\newtheorem{Thm}{Theorem}

\newtheorem{re}{Remark}[section]
\newtheorem{lem}{Lemma}[section]

\newtheorem{theo}[Thm]{Theorem}
 \numberwithin{equation}{section}
 \numberwithin{Thm}{section}

\newcommand{\be}{\begin{equation}}
\newcommand{\ee}{\end{equation}}
\newcommand\bes{\begin{eqnarray}}
\newcommand\ees{\end{eqnarray}}
\newcommand{\bess}{\begin{eqnarray*}}
\newcommand{\eess}{\end{eqnarray*}}

\newcommand\bR{{\mathbb{R}}}

\begin{document}

\title[Principal eigenvalue of an elliptic operator with Large Degenerate Advection]
{Effects of Large Degenerate Advection and Boundary Conditions on the Principal
Eigenvalue and its Eigenfunction of A Linear Second Order
Elliptic Operator}
\author{Rui Peng and\ Maolin Zhou}

\thanks{{\bf R. Peng}: School of Mathematics and Statistics, Jiangsu Normal University,
Xuzhou, 221116, Jiangsu Province, China. Email: {\tt
pengrui\,$\b{}$\,seu@163.com}}

\thanks{{\bf M. Zhou}: Department of Mathematics, School of Science and Technology,
University of New England, Armidale, NSW 2341, Australia. Email: {\tt
zhouutokyo@gmail.com}}

\thanks{R. Peng was partially supported by NSF of China (11271167,
11571200), Natural Science Fund for Distinguished Young Scholars of Jiangsu Province (BK20130002),
and the Priority Academic Program Development of Jiangsu Higher Education Institutions.}
\maketitle

\noindent
{\it Abstract: }
In this article, we study, as the coefficient $s\to\infty$, the asymptotic behavior
of the principal eigenvalue of the eigenvalue problem
 \bes
 -\varphi''(x)-2sm'(x)\varphi'(x)+c(x)\varphi(x)=\lambda(s)\varphi(x),\ \ 0<x<1,
 \nonumber
 \ees
complemented by a general boundary condition.
This problem is relevant to nonlinear propagation phenomena in reaction-diffusion equations.
The main point is that the advection (or drift) term $m$ allows natural degeneracy. For instance,
$m$ can be constant on $[a,b]\subset[0,1]$. Depending on the behavior of $m$
near the neighbourhood of the endpoints $a$ and $b$, the limiting value could be the principal eigenvalue of
\bes
 -\varphi''(x)+c(x)\varphi(x)=\lambda\varphi(x),\ \ a<x<b,
 \nonumber
 \ees
coupled with Dirichlet or Newmann boundary condition at $a$ and $b$. A complete understanding of the limiting behavior of
the principal eigenvalue and its eigenfunction is obtained, and new fundamental effects of large degenerate
advection and boundary conditions on the principal eigenvalue and the principal eigenfunction are revealed.
In one space dimension, the results in the existing literature are substantially improved.

\smallskip


\smallskip


\setlength{\baselineskip}{16pt}{\setlength\arraycolsep{2pt}

\section{Introduction} \setcounter{equation}{0}

In this article, we are concerned with the linear second order elliptic
eigenvalue problem with a general boundary condition in one space dimension:
 \bes
 \left\{\begin{array}{ll}
 \medskip
 \displaystyle
 -\varphi''(x)-2sm'(x)\varphi'(x)+c(x)\varphi(x)=\lambda\varphi(x),\ &0<x<1,\\
 \displaystyle
 -\hbar_1\varphi'(0)+\ell_1\varphi(0)=\hbar_2\varphi'(1)+\ell_2\varphi(1)=0,
 \end{array}
 \right.
 \label{p}
 \ees
where $m\in C^2([0,1]),\,c\in C([0,1])$, and $s\in\bR$ is a parameter
appearing in front of the advection (or drift) term $m$, and the nonnegative constants
$\hbar_i,\,\ell_i\,(i=1,2)$ satisfy $\hbar_i+\ell_i>0\,(i=1,2)$.

It is well known that, given $m,\,c$ and $s$, problem \eqref{p} admits a principal eigenvalue
$\lambda=\lambda_1(s)\in\bR$, which is unique in the sense that only
such an eigenvalue corresponds to a positive eigenfunction $\varphi$
($\varphi$ is also unique up to multiplication). Such a function
$\varphi$ is usually called a principal eigenfunction.

As pointed out by Berestycki, Hamel and Nadirashvili in their remarkable work \cite{BHN},
the qualitative behavior of the eigenvalue problem \eqref{p} usually plays a significant role
in the study of nonlinear propagation phenomena of reaction-diffusion equations.

Therefore, the main goal of this paper is to determine, for a general advection function $m$, as $s\to\infty$,
the limiting behaviors of the principal eigenvalue and its eigenfunction to \eqref{p}.
Problem  \eqref{p} with different types
of degeneracies will be treated in a uniform manner. The obtained results will clearly demonstrate that how the interplay between the degeneracy of
the advection function and the boundary conditions affects the qualitative properties
of the principal eigenvalue and its principal eigenfunction in a substantial way.
As far as we know, the current work seems to be the first to reveal such interesting and fundamental influences.
Besides, we believe that our results will have natural applications to reaction-diffusion equations.

In the rest of the introduction, we first briefly recall the existing works on \eqref{p} in the literature and then state our main findings.

\subsection{Existing studies} Consider the eigenvalue problem with Neumann boundary condition (i.e., $\ell_1=\ell_2=0$ in \eqref{p}):
 \bes
 -\varphi''(x)-2sm'(x)\varphi'(x)+c(x)\varphi(x)=\lambda\varphi(x),\ 0<x<1;\ \ \
 \varphi'(0)=\varphi'(1)=0,
 \label{1.1}
 \ees
and denote $\lambda_1^\mathcal{N}(s)$ and $\varphi$ to be its principal eigenvalue and eigenfunction, respectively.

It is known that $\lambda_1^\mathcal{N}(s)$ enjoys the following variational characterization:
 \bes
 \lambda_1^\mathcal{N}(s)=\min_{\int_0^1 e^{2sm}\varphi^2dx=1}\int_0^1e^{2sm}[(\varphi')^2+c\varphi^2]dx.
 \label{1.2}
 \ees
Using the substitution $w=e^{sm}\varphi$ in \eqref{1.1}, one easily sees that $(\lambda_1^\mathcal{N}(s),w)$ satisfies
 \bes
 \left\{\begin{array}{ll}
 \medskip
 \displaystyle
 -w''(x)+[s^2(m'(x))^2+sm''(x)+c(x)]w(x)=\lambda_1^\mathcal{N}(s)w(x),\ &0<x<1,\\
 \displaystyle
 w'(0)-sw(0)m'(0)=w'(1)-sw(1)m'(1)=0.
 \end{array}
 \right.
 \label{1.3}
 \ees
In light of \eqref{1.2}, the principal eigenvalue $\lambda_1^\mathcal{N}(s)$ can be equivalently characterized by
 \bes
 \lambda_1^\mathcal{N}(s)=\min_{\int_0^1 w^2dx=1}\int_0^1[(w'-swm')^2+cw^2]dx.
 \label{1.4}
 \ees
For later purpose, denote by $w(s,\cdot)$ the positive solution of \eqref{1.3} corresponding to $\lambda_1^\mathcal{N}(s)$, and
normalize it by $\int_0^1w^2(s,x)dx=1$ for each $s>0$. In addition, it is immediately observed that
 $$
 \min_{x\in[0,1]}c(x)\leq\lambda_1^\mathcal{N}(s)\leq\max_{x\in[0,1]}c(x),\ \ \ \forall s\in\bR.
 $$

In \cite{CL1}, Chen and Lou investigated the asymptotic behavior of $\lambda_1^\mathcal{N}(s)$
as $s\to\infty$. To present one of their main results, we need recall some definitions introduced there.
In the one-dimensional setting, Chen and Lou in \cite{CL1} said that

 \begin{itemize}

 \item[$\bullet$] An {\it interior critical point} of the function $m$ is a point $x\in(0,1)$ satisfying $m'(x)=0$,
and an interior critical point $x$ is called {\it non-degenerate} if $m''(x)\not=0$.

 \item[$\bullet$] The boundary points $0$ and $1$ are always called critical, and a boundary critical point $x\in\{0,1\}$
 is called {\it non-degenerate} if either $m'(x)\not=0$ or $m''(x)\not=0$.

 \item[$\bullet$] A {\it point of local maximum} of $m$ is a point $x\in[0,1]$ that satisfies $m(x)\geq m(y)$
 for every $y$ in a small neighborhood of $x$, and there exists some sequence $\{r_j\}_{j=1}^\infty$ of positive numbers such that
  $$
  m(x)>\max_{[0,1]\cap\{x-r_j,\,x+r_j\}}m,\ \ \forall j\geq1,\ \ \lim\limits_{j\to\infty}r_j=0.
  $$

\end{itemize}

As remarked by \cite{CL1}, the reason the authors used such a definition of local maximum is that they had to avoid the occurrence of the situation that
the set of local maximum of $m$ contains some flat piece. Then one main result---Theorem 1 of \cite{CL1} concludes that

 \begin{theo}\label{th1.1} Assume that all critical points of $m$ are non-degenerate. Let $\mathcal{M}$ be the set of points of local maximum of $m$. Then
$\lim\limits_{s\to\infty}\lambda_1^\mathcal{N}(s)=\min_{x\in\mathcal{M}} c(x)$.
\end{theo}

Theorem \ref{th1.1} is of significant importance, and it has found new interesting applications in several
classical reaction-diffusion-advection problems arising from ecology; for example, \cite{CL1,HL,Lam,LNi} to list a few.
Obviously, Theorem \ref{th1.1} deals with only the case that $m$ has finitely many non-degenerate isolated points of local maximum.
Later on, in the companion paper \cite{CL2}, Chen and Lou studied the limiting behavior of $\lambda_1^\mathcal{N}(s)$ when the diffusion and advection rates are both large or small.

In \cite{BHN}, Berestycki, Hamel and Nadirashvili investigated problems like \eqref{p} with
$c(x)\equiv0$ and under either Dirichlet, Neumann or periodic boundary condition in arbitrary space dimensions.
They focussed on the situation that the drift velocity (advection term) ${\bf v}$ is divergence-free, and
established the equivalent connections between the boundedness of the principal eigenvalue with regard to large drift and the existence of the first integrals
of the velocity field ${\bf v}$. As consequences of their results, important influences of large advection or drift on the speed of
propagation of pulsating travelling fronts were revealed there.

As far as the periodic boundary problem is concerned, we assume that $m\in C^2(\bR)$ and $c\in C(\bR)$ and both of them are $1$-periodic
(that is, $f(x)=f(x+1),\,f\in\{m,\,c\},\,\forall x\in\bR$). Then, there is a unique principal eigenvalue to the eigenvalue problem:
 \bes
 -\varphi''(x)-2sm'(x)\varphi'(x)+c(x)\varphi(x)=\lambda\varphi(x),\ x\in\bR;\ \ \
 \varphi(x)=\varphi(x+1),\ x\in\bR.
 \label{1.p1}
 \ees
If one attempts to apply the result of \cite{BHN} to study the limiting behavior of
the principal eigenvalue of problems \eqref{1.1} and \eqref{1.p1} as $s\to\infty$,
the restricted condition on $m$ imposed there now reduces to require that $m$ is constant on $[0,1]$. This is a trivial case.
Here, of our interest is a general nonconstant function $m$ so that spatial heterogeneity of environment can be reflected.

When $\hbar_1=\hbar_2=0$ in \eqref{p}, we have the Dirichlet eigenvalue problem:
 \bes
 -\varphi''(x)-2sm'(x)\varphi'(x)+c(x)\varphi(x)=\lambda\varphi(x),\ 0<x<1;\ \ \
 \varphi(0)=\varphi(1)=0.
 \label{4.1}
 \ees
Then, Theorem 0.3 of \cite{BHN}, one of the main results, only covers the case of $m(x)=ax$ on $[0,1]$ for some constant $a$, and states that
the principal eigenvalue of \eqref{4.1} is bounded as $s\to\infty$ if and only if $a=0$. It turns out that this
is a very special case to be treated in our present work. Indeed, as long as $m'$ changes sign at most finitely many times,
as $s\to\infty$, we are able to derive a necessary
and sufficient condition to guarantee the boundedness of the principal eigenvalue of the general problem \eqref{p}.
Moreover, once the principal eigenvalue is bounded with regard
to large $s$,  the asymptotic behaviors of the principal eigenvalue and its eigenfunction will be precisely given.
See our main results: Theorems \ref{th1.2}, \ref{th1.2a} and \ref{th1.3} below.  We would like to point out that the analysis of \cite{BHN}
seems inapplicable to the general problem \eqref{p}.

In the direction of research on the effect of large advection on the principal eigenvalue, it is worth mentioning a series of impressive work \cite{DEF,DF,Fr},
done by Friedman and his coauthors more than 40 years ago, which concerned the Dirichlet boundary condition case and obtained refined
upper and lower bounds for the principal eigenvalue when the advection coefficient is large. Nevertheless, for \eqref{4.1},
their results seem to apply only to a few special kinds of $m$; for example, $m'(x)$ changes sign on $[0,1]$ at most once.
Regarding other related works, one may refer to \cite{Ec,W} and the references therein.
We further remark that no information of the associated principal eigenfunction was provided in \cite{BHN,DEF,DF,Ec,Fr,W}.

\subsection{Our main results} As mentioned before, the objective of the present paper is to determine, for a general advection function $m$, as $s\to\infty$, the limiting behaviors of the principal eigenvalue and its eigenfunction to problems \eqref{p} and \eqref{1.p1}. Throughout the paper, unless otherwise specified, we always
assume that
 \bes
 \label{a}\ \ \
 \mbox{$m$ is not constant and $m'(x)$\ changes sign at most finitely many times on $[0,1]$}.
 \ees
Hence, here we allow $m$ to have various natural kinds of degeneracy.

Before stating the main results of this paper, we need to classify the set of points of  local maximum points of $m$ and then introduce necessary notation.
We call that

 \begin{itemize}

 \item[$\bullet$] A {\it point of local maximum} of $m$ is a point $x\in[0,1]$ such that there is a small $\epsilon_0>0$ such that $m(x)\geq m(y)$ in $(x-\epsilon_0,x+\epsilon_0)\cap[0,1]$, and such $x$ is said to be an {\it interior point of local maximum} if $x\in(0,1)$;

 \item[$\bullet$] An {\it isolated point of local maximum} of $m$ is a point $x^{I}\in[0,1]$ such that there is a small $\epsilon_0>0$ such that $m(x^{I})>m(x)$ in $(x^{I}-\epsilon_0,x^{I}+\epsilon_0)\cap[0,1]\setminus\{x^{I}\}$;

 \item[$\bullet$] A {\it segment of local maximum} of $m$ is a closed interval $[a,b]\subset[0,1]$ such that there is a small $\epsilon_0>0$ such that
$m$ is constant on the closed interval $[a,b]$, $m$ is monotone (non-increasing or non-decreasing) on $[a-\epsilon_0,a]\cap[0,1]$ and $[b,b+\epsilon_0]\cap[0,1]$, and
for any small $\epsilon>0$, $\{x\in[0,1]:\ m'(x)\not=0\}\cap(a-\epsilon,a)\not=\emptyset$ if $0<a$ and $\{x\in[0,1]:\ m'(x)\not=0\}\cap(b,b+\epsilon)\not=\emptyset$ if $b<1$.

 \end{itemize}
In addition, when $m$ is a $1$-periodic function, we make the convention from now on that the isolated point and segment of local maximum of $m$ is understood
to be restricted to the one period interval $[0,1]$ with the same definitions as above.

Under the assumption \eqref{a}, it is clear that $m$ admits at most finitely many isolated points of local maximum. For later purpose,
we will have to use different notation to distinguish all possible {\it segments of local maximum} of $m$  as follows.

\vskip6pt

$[a^{{i}},b^{{j}}]$ with $0<a^{{i}}<b^{{j}}<1$ and $i,\,j\in\{I,\,D\}$:\ \ $m$ is constant on $[a^{{i}},b^{{j}}]$,
$m$ is non-decreasing (non-increasing, respectively) in $[a^{{i}}-\epsilon_0,a^{{i}}]$ for some small $\epsilon_0>0$
and $\{x\in[0,1]:\ m'(x)>0\}\cap(a^{{i}}-\epsilon,a^{{i}})\not=\emptyset$ ($\{x\in[0,1]:\ m'(x)<0\}\cap(a^{{i}}-\epsilon,a^{{i}})\not=\emptyset$, respectively)
for any small $\epsilon>0$ if $i=I$ (if $i=D$, respectively),
and $m$ is non-decreasing (non-increasing, respectively) in $[b^{{j}},b^{{j}}+\epsilon_0]$ for some small $\epsilon_0>0$
and $\{x\in[0,1]:\ m'(x)>0\}\cap(b^{{j}},b^{{j}}+\epsilon)\not=\emptyset$ ($\{x\in[0,1]:\ m'(x)<0\}\cap(b^{{j}},b^{{j}}+\epsilon)\not=\emptyset$, respectively)
for any small $\epsilon>0$ if $j=I$ (if $j=D$, respectively).

$[0,a^{I}]$ with $0<a^{I}<1$:\ \  $m$ is constant on $[0,a^{I}]$ and $m$ is non-decreasing in $[a^{I},a^{I}+\epsilon_0]$ for some small $\epsilon_0>0$ and $\{x\in[0,1]:\ m'(x)>0\}\cap(a^{{I}},a^{{I}}+\epsilon)\not=\emptyset$ for any small $\epsilon>0$.

$[0,a^{D}]$ with $0<a^{D}<1$:\ \  $m$ is constant on $[0,a^{D}]$ and $m$ is non-increasing in $[a^{D},a^{D}+\epsilon_0]$ for some small $\epsilon_0>0$ and $\{x\in[0,1]:\ m'(x)<0\}\cap(a^{{D}},a^{{D}}+\epsilon)\not=\emptyset$ for any small $\epsilon>0$.

$[a^{I},1]$ with $0<a^{I}<1$:\ \  $m$ is constant on $[a^{I},1]$ and $m$ is non-decreasing in $[a^{I}-\epsilon_0,a^{I}]$ for some small $\epsilon_0>0$ and $\{x\in[0,1]:\ m'(x)>0\}\cap(a^{{I}}-\epsilon,a^{I})\not=\emptyset$ for any small $\epsilon>0$.

$[a^{D},1]$ with $0<a^{D}<1$:\ \  $m$ is constant on $[a^{D},1]$ and $m$ is non-increasing in $[a^{D}-\epsilon_0,a^{D}]$ for some small $\epsilon_0>0$ and $\{x\in[0,1]:\ m'(x)<0\}\cap(a^{{D}}-\epsilon,a^{D})\not=\emptyset$ for any small $\epsilon>0$.

\vskip3pt
In the above, the capital letters $I$ and $D$ represent increasing (i.e., non-decreasing)
and decreasing (i.e., non-increasing), respectively; for instance, $[a^{{I}},b^{{D}}]$ means that $m$ increases in a left neighbourhood of $a^{{I}}$ and decreases in a right
neighbourhood of $b^{{D}}$ while $m$ is constant on $[a^{{I}},b^{{D}}]$.

Under the assumption \eqref{a}, it is noted that $m$ admits at most finitely many segments $[a^{{I}},b^{{D}}]$ and $[a^{{D}},b^{{I}}]$ of local maximum,
and it has at most countably many segments $[a^{{I}},b^{{I}}]$ and $[a^{{D}},b^{{D}}]$ of local maximum. We need more notation as follows.

Given a closed interval $[a,b]\subset[0,1]$ with $0<a<b<1$ and $i,j\in\{\mathcal{N},\, \mathcal{D}\}$, we denote by $\lambda_1^{{ij}}(a,b)$
the principal eigenvalue of the elliptic eigenvalue problem:
 \bes
 -\varphi''(x)+c(x)\varphi(x)=\lambda\varphi(x),\ a<x<b
 \nonumber
 \ees
subject to Neumann boundary condition (Dirichlet boundary condition, respectively) at the left boundary point $a$ if $i=\mathcal{N}$ (if $i=\mathcal{D}$, respectively), and Neumann boundary condition (Dirichlet boundary condition, respectively) at the right boundary point $b$ if $j=\mathcal{N}$ (if $j=\mathcal{D}$, respectively).

Given $[0,b]\subset[0,1]$ with $0<b<1$ and $i\in\{\mathcal{N},\, \mathcal{D}\}$, denote by $\lambda_1^{{\mathcal{R}i}}(0,b)$
the principal eigenvalue of
 \bes
 -\varphi''(x)+c(x)\varphi(x)=\lambda\varphi(x),\ 0<x<b;\ \ \ -\hbar_1\varphi'(0)+\ell_1\varphi(0)=0
 \nonumber
 \ees
with Neumann boundary condition (Dirichlet boundary condition, respectively ) at $b$ if $i=\mathcal{N}$ (if $i=\mathcal{D}$, respectively).

Given $[a,1]\subset[0,1]$ with $0<a<1$ and $i\in\{\mathcal{N},\, \mathcal{D}\}$, denote by $\lambda_1^{{i\mathcal{R}}}(a,1)$
the principal eigenvalue of
 \bes
 -\varphi''(x)+c(x)\varphi(x)=\lambda\varphi(x),\ a<x<1; \ \ \ \hbar_2\varphi'(1)+\ell_2\varphi(1)=0
 \nonumber
 \ees
with Neumann boundary condition (Dirichlet boundary condition, respectively ) at $a$ if $i=\mathcal{N}$ (if $i=\mathcal{D}$, respectively).


According to the assumption \eqref{a}, it is obviously seen that the set of points of local maximum of $m$ can be represented by
 \bes
 \mathcal{M}=\cup_{i=1}^9\mathcal{M}_i,
 \label{max}
 \ees
in which
 \bes
 \left.\begin{array}{lll}
 \medskip
 \displaystyle
 \mathcal{M}_1=\{x^{I}_i\}_{i=1}^{h_1},\ \ \ \ \mathcal{M}_2=\{[a_i^{{I}},b_i^{{I}}]\}_{i=1}^{h_2},
 \ \ \ \ \mathcal{M}_3=\{[a_i^{{I}},b_i^{{D}}]\}_{i=1}^{h_3},\ \ \ \ \mathcal{M}_4=\{[a_i^{{D}},b_i^{{I}}]\}_{i=1}^{h_4},\\
 \medskip
 \displaystyle
 \mathcal{M}_5=\{[a_i^{{D}},b_i^{{D}}]\}_{i=1}^{h_5},\ \ \mathcal{M}_6\subset\{[0,a^{{I}}]\},\ \
 \mathcal{M}_7\subset\{[0,a^{{D}}]\},\ \ \mathcal{M}_8\subset\{[a^{{I}},1]\},\ \ \mathcal{M}_9\subset\{[a^{{D}},1]\},
 \end{array}
 \right.
 \nonumber
 \ees
where $h_1,\,h_3,\,h_4$ are finite integers while $h_2,\,h_5$ may be finite integers or infinity. We allow some $\mathcal{M}_i$ to be empty.
According to our definitions, $x\cap y=\emptyset$ for any $x,\,y\in\mathcal{M}_i,\,x\not=y,\,i\in\{2,3,4,5\}$, and
$x\cap y=\emptyset$ for any $x\in\mathcal{M}_i,\,y\in\mathcal{M}_j$ if $i\not=j$. Moreover, either $\mathcal{M}_6$ or
$\mathcal{M}_7$ or both of them must be an empty set, and the same is true for $\mathcal{M}_8$ and $\mathcal{M}_9$. Though
some $\mathcal{M}_i$ may be empty, it is apparent that $\mathcal{M}\not=\emptyset$.

For simplicity, let us also set
 \bes
 \left.\begin{array}{lll}
 \medskip
 \displaystyle
 \mathfrak{L}&=&
 \displaystyle
 \min\Big\{\min\{\lambda_1^{\mathcal{ND}}(a^I_i,b^I_i):\ [a^I_i,b^I_i]\in\mathcal{M}_2\},\ \
 \min\{\lambda_1^{\mathcal{NN}}(a^I_i,b^D_i):\ [a^I_i,b^D_i]\in\mathcal{M}_3\},\\
 \medskip
 &&\ \ \ \ \ \ \ \
 \displaystyle
 \min\{\lambda_1^{\mathcal{DD}}(a^D_i,b^I_i):\ [a^D_i,b^I_i]\in\mathcal{M}_4\},\ \
 \min\{\lambda_1^{\mathcal{DN}}(a^D_i,b^D_i):\ [a^D_i,b^D_i]\in\mathcal{M}_5\}\Big\}.
 \end{array}
 \right.
 \nonumber
 \ees
We want to stress that $\min\{\lambda_1^{\mathcal{ND}}(a^I_i,b^I_i):\ [a^I_i,b^I_i]\in\mathcal{M}_2\}$ and
$\min\{\lambda_1^{\mathcal{DN}}(a^D_i,b^D_i):\ [a^D_i,b^D_i]\in\mathcal{M}_5\}$
are well defined since $\lambda_1^{\mathcal{ND}}(a^I_i,b^I_i)\to\infty$ when $b^I_i-a^I_i\to0$ and $\lambda_1^{\mathcal{DN}}(a^D_i,b^D_i)\to\infty$ when
$b^D_i-a^D_i\to0$.

Our first result concerns the limiting behavior of the principal eigenvalue $\lambda_1(s)$ for problem \eqref{p},
and reads as follows.

 \begin{theo}\label{th1.2} Assume that the set $\mathcal{M}$ of points of local maximum of $m$ is given by \eqref{max}. Then the following assertions hold.

 \begin{itemize}

\item[(i)] $\lim\limits_{s\to\infty}\lambda_1(s)=\infty$ if and only if

{\rm{(i-1)}}\ $\mathcal{M}=\mathcal{M}_1\subset\{0,1\}$ when $\ell_1>0$ and $\ell_2>0$;\

{\rm{(i-2)}}\ $\mathcal{M}=\mathcal{M}_1=\{0\}$ when $\ell_1>0$ and $\ell_2=0$;\

{\rm{(i-3)}}\ $\mathcal{M}=\mathcal{M}_1=\{1\}$ when $\ell_1=0$ and $\ell_2>0$.

\item[(ii)] If $\lim\limits_{s\to\infty}\lambda_1(s)<\infty$ , then
 \bes
 \left.\begin{array}{lll}
 \medskip
 \displaystyle
 \lim\limits_{s\to\infty}\lambda_1(s)&=&
 \medskip
 \displaystyle
 \min\Big\{\mathfrak{L},\ \ \min\{c(x):\ x\in\mathcal{M}_1\setminus\{0,\,1\}\},\ \ \lambda_1^{\mathcal{RD}}(0,a^I)\ (\mbox{if}\ [0,a^I]\in\mathcal{M}_6),\\
 &&\ \ \ \ \ \ \
 \medskip
 \displaystyle
 \lambda_1^{\mathcal{RN}}(0,a^D) \ (\mbox{if}\ [0,a^D]\in\mathcal{M}_7),\ \ \lambda_1^{\mathcal{NR}}(a^I,1)\ (\mbox{if}\ [a^I,1]\in\mathcal{M}_8),\\
 &&\ \ \ \ \ \ \
 \medskip
 \displaystyle
 \lambda_1^{\mathcal{DR}}(a^D,1)\ (\mbox{if}\ [a^D,1]\in\mathcal{M}_9),\ \ c(0)\ (\mbox{if\ $0\in\mathcal{M}_1$ and $\ell_1=0$}),\\
 &&\ \ \ \ \ \ \
 \displaystyle
 c(1)\ (\mbox{if\ $1\in\mathcal{M}_1$ and $\ell_2=0$})\, \Big\}.
 \end{array}
 \right.
 \nonumber
 \ees

\end{itemize}
\end{theo}

\begin{re}\label{r-result} Concerning Theorem \ref{th1.2}, we would like to make the following comments.

\begin{itemize}

\item[(i)] In comparison with Theorem \ref{th1.1}, when the Neumann problem \eqref{1.1} is concerned, Theorem \ref{th1.2} covers the case that $m$ has segments of local maximum;
and even if $m$ has isolated points of local maximum, unlike Theorem \ref{th1.1}, we do not impose non-degeneracy condition on those points.

\item[(ii)] From Theorem \ref{th1.2}, we can see that when $0$ or $1$ is an isolated point of local maximum of $m$, the limiting value of the principal eigenvalue, when finite,
is affected by such a boundary point only if the Neumann boundary condition is prescribed there.

\item[(iii)] When $\ell_1+\ell_2>0$, in view of Theorem \ref{th1.2}, we have the following observations.

\begin{itemize}

\item[(iii-1)] If $\ell_1,\,\ell_2>0$, then $\lim\limits_{s\to\infty}\lambda_1(s)=\infty$ if and only if
either $m$ is strictly increasing or strcitly decreasing on $[0,1]$, or $m$ is strictly decreasing on $[0,x_0]$ while strictly increasing on $[x_0,1]$ for some $x_0\in(0,1)$.

\item[(iii-2)] If $\ell_1=0,\,\ell_2>0$, then $\lim\limits_{s\to\infty}\lambda_1(s)=\infty$ if and only if $m$ is strictly increasing $[0,1]$.

\item[(iii-3)] If $\ell_1>0,\,\ell_2=0$, then $\lim\limits_{s\to\infty}\lambda_1(s)=\infty$ if and only if $m$ is strictly decreasing $[0,1]$.

\end{itemize}

\item[(iv)] For the Dirichlet eigenvalue problem \eqref{4.1}, the main results of \cite{DEF} tell us that if
$m'(x)>0$ or $m'(x)<0$ on $[0,1]$, then there exists a constant $\vartheta>1$ such that
 $$
 \mbox{$s^2/\vartheta\leq\lambda_1(s)\leq \vartheta s^2$\ \ \ as\ $s\to\infty$},
 $$
and if $m'(x)<0$ in $[0,x_0)$, $m'(x)>0$ in $(x_0,1]$, and $|x-x_0|^{1+\nu}/\sigma\leq|m'(x)|\leq\sigma|x-x_0|^{1+\nu},\,\forall x\in[0,1]$
for some $x_0\in(0,1)$ and constants $\sigma>1,\,\nu>0$, then there exists a constant $\vartheta>1$ such that
 $$
 \mbox{$s^{2\over{\nu+1}}/\vartheta\leq\lambda_1(s)\leq \vartheta s^{2\over{\nu+1}}$\ \ \ as\ $s\to\infty$.}
 $$
On the other hand, \cite{Fr} showed that if $c(x)=c$ is constant, $m'(0)>0$, $m'(1)<0$, and $m$ has only one isolated interior point of local maximum in $(0,1)$, then
$\lambda_1(s)=c+O(se^{-s})$ as $s\to\infty$. However, the asymptotic growth rate of $\lambda_1(s)$ is not known in general.

\end{itemize}

\end{re}

We next consider the periodic eigenvalue problem \eqref{1.p1}. Without loss of generality, we assume that $m'(0)>0$ (and so $m'(1)>0$)
due to $1$-periodicity of $m$. Then the following result holds.

\begin{theo}\label{th1.2a} Assume that $m'(0)>0$ and the set of points of local maximum of $m$ is given by \eqref{max},
and denote by $\lambda_1^\mathcal{P}(s)$ the principal eigenvalue of \eqref{1.p1}. Then
 \bes
 \lim\limits_{s\to\infty}\lambda_1^\mathcal{P}(s)=
 \min\Big\{\mathfrak{L},\ \ \min\{c(x):\ x\in\mathcal{M}_1\}\Big\}.
 \nonumber
 \ees

\end{theo}

We now turn our attention to the limiting profile of the principal eigenfunction.
In what follows, for sake of simplicity we only state the result for problem \eqref{1.1}; for the general
problem \eqref{p} and the periodic problem \eqref{1.p1}, the analogous result remains true.

As in \cite{CL1}, we define
 $$
 \lambda^*=\limsup_{s\to\infty}\lambda_1^\mathcal{N}(s),\ \ \  \lambda_*=\liminf_{s\to\infty}\lambda_1^\mathcal{N}(s).
 $$
Recall that $w(s,\cdot)$ is the positive solution of \eqref{1.3} corresponding to the principal eigenvalue
$\lambda_1^\mathcal{N}(s)$ with the normalization $\int_0^1w^2(s,x)dx=1$ for each $s>0$. It is well known that
the sequence $\{w^2(s,\cdot)\}_{s>0}$ is weakly compact in the space $L^1(0,1)$. This implies that there exists a sequence
$\{s_j\}_{j=1}^\infty$ satisfying $s_j\to\infty$ as $j\to\infty$, such that
 \bes
 \lim\limits_{j\to\infty}\int_0^1w^2(s_j,x)\zeta(x)dx=\int_{[0,1]}\zeta(x)\mu(dx),\ \ \forall \zeta\in C([0,1])
 \label{1.7}
 \ees
for a certain probability measure $\mu$. Therefore, from \eqref{1.4} and \eqref{1.7}, it follows that
 \bes
 \lambda_*\geq\lim\limits_{j\to\infty}\int_0^1w^2(s_j,x)c(x)dx=\int_{[0,1]}c(x)\mu(dx).
 \label{1.8}
 \ees

\vskip6pt
We first see from Lemmas \ref{l2.3} and \ref{l2.4} and their proofs that, roughly speaking, in the set $E\subset[0,1]$ where $m$ has no local maximum, $\mu(E)=0$ and
$w(s,\cdot)\to0$ in $E$ as $s\to\infty$. We are now interested in the asymptotic behavior of the principal eigenfunction
$w(s,\cdot)$ in an isolated point or segment of local maximum of $m$ carrying a positive Radon measure of $\mu$. More precisely, we have

\begin{theo}\label{th1.3} Let $w(s,\cdot)$, normalized  by $\int_0^1w^2(s,x)dx=1$ for each $s>0$, be the principal eigenfunction of \eqref{1.1}
corresponding to $\lambda_1^\mathcal{N}(s)$, and $\mu$ be the Radon measure defined through the sequence $\{s_j\}$ in \eqref{1.7}. The following assertions hold.

 \begin{enumerate}
 \item[(1)] Assume that $x_0\in\mathcal{M}_1$ and satisfies
 $$
 m^{(k)}(x_0)=0,\ \ \forall 1\leq k\leq k^*-1,\ \ \mbox{and}\ \ m^{(k^*)}(x_0)\not=0,
 $$
for some integer $k^*\geq2$, and $\mu(\{x_0\})>0$. Then $\lim\limits_{s\to\infty}\lambda_1^\mathcal{N}(s)=c(x_0)$. Moreover, up to a subsequence of $\{s_j\}$, we have

   \begin{enumerate}
 \item[(1-i)] If $x_0\in(0,1)$, then
 $$
 \mbox{$\mu(\{x_0\})^{-{1\over2}}s^{{1\over{2k^*}}}w(s,x_0+s^{-{1\over{k^*}}}y)\to W^*$ locally uniformly in $\bR$,}
 $$
 where $W^*$ with $\int_{-\infty}^\infty (W^*)^2(y)dy=1$ is a positive solution of the linear ODE equation
 \bes
 (W^*)''(y)=((m^{(k^*)}(x_0))^2y^{2(k^*-1)}+m^{(k^*)}(x_0)y^{k^*-2})W^*(y),\ \ \ y\in\bR.
 \nonumber
 \ees

 \item[(1-ii)] If $x_0=0$ or $x_0=1$, then
 $$
 \mbox{$\mu(\{x_0\})^{-{1\over2}}s^{{1\over{2k^*}}}w(s,x_0+s^{-{1\over{k^*}}}y)\to W_*$ locally uniformly in $\bR_*$,}
 $$
 where $\bR_*=(0,\infty)$ if $x_0=0$ and $\bR_*=(-\infty,0)$ if $x_0=1$, and $W_*$ with $\int_{\bR_*}(W_*)^2(y)dy=1$
 is a positive solution of the linear ODE equation
 \bes
 (W_*)''(y)=((m^{(k^*)}(x_0))^2y^{2(k^*-1)}+m^{(k^*)}(x_0)y^{k^*-2})W_*(y),\ \ \ y\in\bR_*.
 \nonumber
 \ees

   \end{enumerate}

 \item[(2)] Assume that $[a,b]\in\mathcal{M}_2$ and $\mu((a,b))>0$. Then $\lim\limits_{s\to\infty}\lambda_1^\mathcal{N}(s)=\lambda_1^{\mathcal{ND}}(a,b)$, and
 $w(s,\cdot)\to\varphi_1$ in $C^1([a,b])$, where $\varphi_1$ is an eigenfunction corresponding to $\lambda_1^{\mathcal{ND}}(a,b)$.

 \item[(3)] Assume that $[a,b]\in\mathcal{M}_3$ and $\mu((a,b))>0$. Then $\lim\limits_{s\to\infty}\lambda_1^\mathcal{N}(s)=\lambda_1^{\mathcal{NN}}(a,b)$, and
 $w(s,\cdot)\to\varphi_1$ in $C^1([a,b])$, where $\varphi_1$ is an eigenfunction corresponding to $\lambda_1^{\mathcal{NN}}(a,b)$.

 \item[(4)] Assume that $[a,b]\in\mathcal{M}_4$ and $\mu((a,b))>0$. Then $\lim\limits_{s\to\infty}\lambda_1^\mathcal{N}(s)=\lambda_1^{\mathcal{DD}}(a,b)$, and
 $w(s,\cdot)\to\varphi_1$ in $C^1([a,b])$, where $\varphi_1$ is an eigenfunction corresponding to $\lambda_1^{\mathcal{DD}}(a,b)$.

 \item[(5)] Assume that $[a,b]\in\mathcal{M}_5$ and $\mu((a,b))>0$. Then $\lim\limits_{s\to\infty}\lambda_1^\mathcal{N}(s)=\lambda_1^{\mathcal{DN}}(a,b)$, and
 $w(s,\cdot)\to\varphi_1$ in $C^1([a,b])$, where $\varphi_1$ is an eigenfunction corresponding to $\lambda_1^{\mathcal{DN}}(a,b)$.

 \item[(6)] Assume that $[0,a^I]\in\mathcal{M}_6$ and $\mu([0,a^I))>0$. Then $\lim\limits_{s\to\infty}\lambda_1^\mathcal{N}(s)=\lambda_1^{\mathcal{ND}}(0,a^I)$, and
 $w(s,\cdot)\to\varphi_1$ in $C^1([0,a^I])$, where $\varphi_1$ is an eigenfunction corresponding to $\lambda_1^{\mathcal{ND}}(0,a^I)$.

 \item[(7)] Assume that $[0,a^D]\in\mathcal{M}_7$ and $\mu([0,a^D))>0$. Then $\lim\limits_{s\to\infty}\lambda_1^\mathcal{N}(s)=\lambda_1^{\mathcal{NN}}(0,a^D)$, and
 $w(s,\cdot)\to\varphi_1$ in $C^1([0,a^D])$, where $\varphi_1$ is an eigenfunction corresponding to $\lambda_1^{\mathcal{NN}}(0,a^D)$.

 \item[(8)] Assume that $[a^I,1]\in\mathcal{M}_8$ and $\mu((a^I,1])>0$. Then $\lim\limits_{s\to\infty}\lambda_1^\mathcal{N}(s)=\lambda_1^{\mathcal{NN}}(a^I,1)$, and
 $w(s,\cdot)\to\varphi_1$ in $C^1([a^I,1])$, where $\varphi_1$ is an eigenfunction corresponding to $\lambda_1^{\mathcal{NN}}(a^I,1)$.

 \item[(9)] Assume that $[a^D,1]\in\mathcal{M}_9$ and $\mu((a^D,1])>0$. Then $\lim\limits_{s\to\infty}\lambda_1^\mathcal{N}(s)=\lambda_1^{\mathcal{DN}}(a^D,1)$, and
 $w(s,\cdot)\to\varphi_1$ in $C^1([a^D,1])$, where $\varphi_1$ is an eigenfunction corresponding to $\lambda_1^{\mathcal{DN}}(a^D,1)$.

 \end{enumerate}

\end{theo}

We remark that if $k^*=2$, Theorem \ref{th1.3} is reduced to Theorem 2 of \cite{CL1} in the one dimension case,
and the unique positive solution $W^*$ and $W_*$ can be explicitly given (see Theorem 2 of \cite{CL1} for the details).

To obtain the results stated above, we mainly follow the approach of \cite{CL1}.
However, in doing so, several new ideas and techniques will have to be introduced in order to overcome a number of highly nontrivial
difficulties caused by the degeneracy of $m$.

Roughly speaking, our strategy consists of two main steps.
As a first step, we establish $\limsup\limits_{s\to\infty}\lambda_1(s)$ by constructing suitable testing functions
(due to the variational characterization of $\lambda_1(s)$). In the second step of yielding $\liminf\limits_{s\to\infty}\lambda_1(s)$, we first need to
establish a refined description of the support of the probability measure $\mu$, showing that
$\mu([0,1]\setminus\mathcal{M})=0$. Then, combined with the variational characterization of $\lambda_1(s)$ again,
among other ingredients, we can derive $\liminf\limits_{s\to\infty}\lambda_1(s)$, which coincides with $\limsup\limits_{s\to\infty}\lambda_1(s)$.
The asymptotic profile of the eigenfunction $w(s,x)$ is determined by using some local analysis at an isolated point of local maximum of $m$,
and by elliptic regularity theory in a segment of local maximum of $m$.

In \cite{CL1}, under the assumption of Theorem \ref{th1.1}, for problem \eqref{1.1}, Chen and Lou asked if the support of the probability measure $\mu$ is precisely given
by the set $\mathcal{M}$ of points of local maximum. Theorems \ref{th1.2}, \ref{th1.3} here and their proofs show that this is not the case in general.
As a matter of fact, the support of $\mu$ consists of the set
through which the limit $\lim\limits_{s\to\infty}\lambda_1^\mathcal{N}(s)$ is attained, and conversely, the limit is attained on the support of $\mu$ in the same sense
as explained in Remark \ref{r-a}. Such comments also apply to the general problem \eqref{p} and the periodic problem \eqref{1.p1}.

To illustrate the main results obtained in this paper, we shall look at the following three typical examples.

Example 1: $m$ is strictly decreasing on $[0,x_1]\cup[x_3,x_4]$, strictly increasing on $[x_1,x_2]\cup[x_4,1]$ and constant on $[x_2,x_3]$ (See Figure 1).
Let $x_2,\,x_3$ shrink to one point $x_0$ so that
$m$ is strictly decreasing on $[0,x_1]\cup[x_0,x_4]$ and strictly increasing on $[x_1,x_0]\cup[x_4,1]$ (See Figure 2).

\begin{figure}[htbp]\label{figure1}
\centering {\includegraphics[height=1.35in]{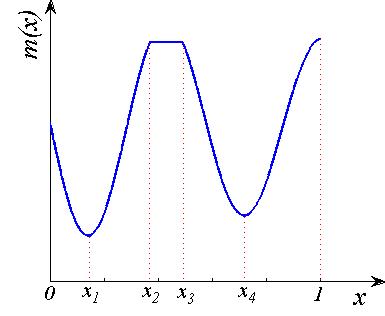}}\ \ \ \ \ \ \  \ \ \ \ \ \ \ \ \ \ \ \  \ \ \
{\includegraphics[height=1.35in]{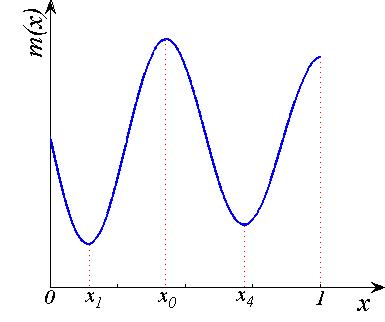}}

{Figure 1 \ \ \ \ \ \ \ \ \ \ \ \ \ \ \ \ \ \ \ \ \ \ \ \ \  \ \  \ \  \ \ \  \ \ \ \ \ \ \ \ \ \ Figure 2}
\end{figure}

Denote by $\lambda_1^{\mathcal{RR}}(s)$ ($\lambda_1^{\mathcal{NR}}(s)$; $\lambda_1^{\mathcal{RN}}(s)$, respectively)
the principal eigenvalue of \eqref{p} with $\ell_1,\,\ell_2>0$ ($\ell_1=0,\,\ell_2>0$; $\ell_1>0,\,\ell_2=0$, respectively), and let $\mu$ be the
probability measure corresponding to the normalized principal eigenfunction $w(s,\cdot)$ of the associated eigenvalue
problem as defined for the Neumann problem \eqref{1.1} and \eqref{1.3}.

Assume that $m$ is given as in Figure 1, we have

\begin{itemize}

\item[(1)] $\lim\limits_{s\to\infty}\lambda_1^{\mathcal{N}}(s)=\min\{c(0),\ c(1),\ \lambda_1^\mathcal{NN}(x_2,x_3)\}$,
$\mu((0,x_2]\cup[x_3,1))=0$ and $\mu(\{0,\,1\}\cup(x_2,x_3))=1$.

\item[(2)]

$\lim\limits_{s\to\infty} \lambda_1^{\mathcal{RR}}(s)=\lambda_1^\mathcal{NN}(x_2,x_3)$, and
$\mu([0,x_2]\cup[x_3,1])=0$ and $\mu((x_2,x_3))=1$.

\item[(3)]
$\lim\limits_{s\to\infty}\lambda_1^{\mathcal{NR}}(s)=\min\{c(0),\ \lambda_1^\mathcal{NN}(x_2,x_3)\}$,
$\mu((0,x_2]\cup[x_3,1])=0$ and $\mu(\{0\}\cup(x_2,x_3))=1$.

\item[(4)]
$\lim\limits_{s\to\infty}\lambda_1^{\mathcal{RN}}(s)=\min\{c(1),\ \lambda_1^\mathcal{NN}(x_2,x_3)\}$,
$\mu([0,x_2]\cup[x_3,1))=0$ and $\mu(\{1\}\cup(x_2,x_3))=1$.

\end{itemize}

When $x_2$ and $x_3$ shrink to one point $x_0$ as shown in Figure 2, we have
 \begin{itemize}

\item[(1')] $\lim\limits_{s\to\infty}\lambda_1^{\mathcal{N}}(s)=\min\{c(0),\ c(1),\ c(x_0)\}$,
$\mu((0,x_0)\cup(x_0,1))=0$ and $\mu(\{0,\,1,\,x_0\})=1$.

\item[(2')]

$\lim\limits_{s\to\infty} \lambda_1^{\mathcal{RR}}(s)=c(x_0)$, $\mu([0,x_0)\cup(x_0,1])=0$ and $\mu(\{x_0\})=1$.

\item[(3')]
$\lim\limits_{s\to\infty}\lambda_1^{\mathcal{NR}}(s)=\min\{c(0),\ c(x_0)\}$,
$\mu((0,x_0)\cup(x_0,1])=0$ and $\mu(\{0,\,x_0\})=1$.

\item[(4')]
$\lim\limits_{s\to\infty}\lambda_1^{\mathcal{RN}}(s)=\min\{c(1),\ c(x_0)\}$,
$\mu([0,x_0]\cup(x_0,1))=0$ and $\mu(\{x_0,\,1\})=1$.

\end{itemize}

\vskip6pt
We note that $\lambda_1^\mathcal{NN}(x_2,x_3)$ converges to $c(x_0)$ as $x_2,\,x_3$ shrinks to the point $x_0$.
Thus, our result coincides with Theorem \ref{th1.1} obtained by Chen and Lou \cite{CL1}; however, we do not require
non-degeneracy of $m$ at $x_0$.

\vskip6pt Example 2: $m$ is strictly decreasing on $[0,x_1]$, strictly increasing on $[x_2,1]$ and constant on $[x_1,x_2]$ (Figure 3).
Let $x_1,\,x_2$ shrink to one point $x_0$ so that
$m$ is strictly decreasing on $[0,x_0]$ and strictly increasing on $[x_0,1]$ (Figure 4). Assume that $m$ is given as in Figure 3, we have

\begin{figure}[htbp]\label{figure2}
\centering {\includegraphics[height=1.35in]{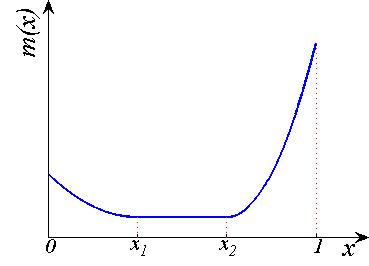}}\ \ \ \ \ \ \  \ \ \ \ \ \ \ \ \ \ \ \  \ \ \
{\includegraphics[height=1.35in]{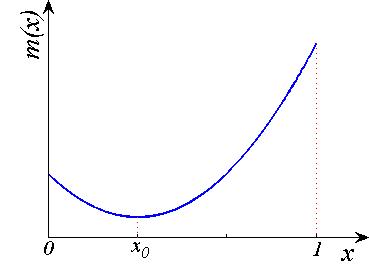}}

{Figure 3 \ \ \ \ \ \ \ \ \ \ \ \ \ \ \ \ \ \ \ \ \ \ \ \ \  \ \  \ \  \ \ \  \ \ \ \ \ \ \ \ \ \ \ \ \ \ \ \ Figure 4}
\end{figure}

\begin{itemize}

\item[(1)] $\lim\limits_{s\to\infty}\lambda_1^{\mathcal{N}}(s)=\min\{c(0),\ c(1),\ \lambda_1^\mathcal{DD}(x_1,x_2)\}$,
$\mu((0,x_1]\cup[x_2,1))=0$ and $\mu(\{0,\,1\}\cup(x_1,x_2))=1$.

\item[(2)]

$\lim\limits_{s\to\infty} \lambda_1^{\mathcal{RR}}(s)=\lambda_1^\mathcal{DD}(x_1,x_2)$,
$\mu([0,x_1]\cup[x_2,1])=0$ and $\mu((x_1,x_2))=1$.

\item[(3)]
$\lim\limits_{s\to\infty}\lambda_1^{\mathcal{NR}}(s)=\min\{c(0),\ \lambda_1^\mathcal{DD}(x_1,x_2)\}$,
$\mu((0,x_1]\cup[x_2,1])=0$ and $\mu(\{0\}\cup(x_1,x_2))=1$.

\item[(4)]
$\lim\limits_{s\to\infty}\lambda_1^{\mathcal{RN}}(s)=\min\{c(1),\ \lambda_1^\mathcal{DD}(x_1,x_2)\}$,
$\mu([0,x_1]\cup[x_2,1))=0$ and $\mu(\{1\}\cup(x_1,x_2))=1$.

\end{itemize}

When $x_1$ and $x_2$ shrink to one point $x_0$ as shown in Figure 4, we have
 \begin{itemize}

\item[(1')] $\lim\limits_{s\to\infty}\lambda_1^{\mathcal{N}}(s)=\min\{c(0),\ c(1)\}$,
$\mu((0,1))=0$ and $\mu(\{0,\,1\})=1$.

\item[(2')]

$\lim\limits_{s\to\infty} \lambda_1^{\mathcal{RR}}(s)=\infty$.

\item[(3')]
$\lim\limits_{s\to\infty}\lambda_1^{\mathcal{NR}}(s)=c(0)$,
$\mu((0,1])=0$ and $\mu(\{0\})=1$.

\item[(4')]
$\lim\limits_{s\to\infty}\lambda_1^{\mathcal{RN}}(s)=c(1)$,
$\mu([0,1))=0$ and $\mu(\{1\})=1$.

\end{itemize}

\vskip6pt Example 3: $m$ is strictly increasing on $[0,x_1]\cup[x_2,1]$ and constant on $[x_1,x_2]$ (Figure 5).
Let $x_1,\,x_2$ shrink to one point $x_0$ so that $m$ is strictly increasing on $[0,1]$ (Figure 6). Assume that $m$ is given as in Figure 5, we have

\begin{figure}[htbp]\label{figure3}
\centering {\includegraphics[height=1.35in]{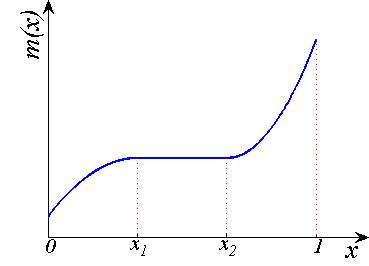}}\ \ \ \ \ \ \  \ \ \ \ \ \ \ \ \ \ \ \  \ \ \
{\includegraphics[height=1.35in]{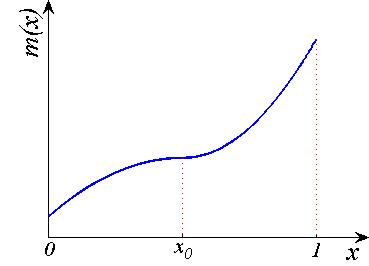}}

{Figure 5 \ \ \ \ \ \ \ \ \ \ \ \ \ \ \ \ \ \ \ \ \ \ \ \ \  \ \  \ \ \ \ \ \ \ \ \ \ \ \  \ \ \ \ \ \ \ \ \ \ Figure 6}
\end{figure}

\begin{itemize}

\item[(1)] $\lim\limits_{s\to\infty}\lambda_1^{\mathcal{N}}(s)=\min\{c(1),\ \lambda_1^\mathcal{ND}(x_1,x_2)\}$,
$\mu([0,x_1]\cup[x_2,1))=0$ and $\mu(\{1\}\cup(x_1,x_2))=1$.

\item[(2)]

$\lim\limits_{s\to\infty} \lambda_1^{\mathcal{RR}}(s)=\lambda_1^\mathcal{ND}(x_1,x_2)$, and
$\mu([0,x_1]\cup[x_2,1])=0$ and $\mu((x_1,x_2))=1$.

\item[(3)]
$\lim\limits_{s\to\infty}\lambda_1^{\mathcal{NR}}(s)=\lambda_1^\mathcal{ND}(x_1,x_2)$,
$\mu([0,x_1]\cup[x_2,1])=0$ and $\mu((x_1,x_2))=1$.

\item[(4)]
$\lim\limits_{s\to\infty}\lambda_1^{\mathcal{RN}}(s)=\min\{c(1),\ \lambda_1^\mathcal{ND}(x_1,x_2)\}$,
$\mu([0,x_1]\cup[x_2,1))=0$ and $\mu(\{1\}\cup(x_1,x_2))=1$.

\end{itemize}

When $x_1$ and $x_2$ shrink to one point $x_0$ as shown in Figure 6, we have
 \begin{itemize}

\item[(1')] $\lim\limits_{s\to\infty}\lambda_1^{\mathcal{N}}(s)=c(1)$,
$\mu([0,1))=0$ and $\mu(\{1\})=1$.

\item[(2')]

$\lim\limits_{s\to\infty} \lambda_1^{\mathcal{RR}}(s)=\infty$.

\item[(3')]
$\lim\limits_{s\to\infty}\lambda_1^{\mathcal{NR}}(s)=\infty$,
$\mu([0,1))=0$ and $\mu(\{1\})=1$.

\item[(4')]
$\lim\limits_{s\to\infty}\lambda_1^{\mathcal{RN}}(s)=c(1)$,
$\mu([0,1))=0$ and $\mu(\{1\})=1$.

\end{itemize}

\vskip6pt
We also note that, in Examples 2 and 3, $\lambda_1^\mathcal{DD}(x_1,x_2)\to\infty$ and $\lambda_1^\mathcal{ND}(x_1,x_2)\to\infty$
as $x_1,\,x_2$ shrinks to the point $x_0$.
Again, our above results are consistent with Theorem \ref{th1.1} due to Chen and Lou \cite{CL1}; though $m'(x_0)=0$, we do not require
non-degeneracy of $m$ at $x_0$ (that is, $m''(x_0)$ may vanish). Furthermore, in each of the above three examples, the support of $\mu$ consists of the set
through which the limiting value of the principal eigenvalue is attained,
and conversely, the limit is attained on the support of $\mu$ in the same sense
as interpreted in Remark \ref{r-a}.

To end the introduction, we shall use two simple examples to hint
the interesting impact of {\it oscillating behavior} of $m$ on the principal eigenvalue.

\vskip4pt
Example A: Assume that the set of points of local maximum of
$m$, denoted by $\mathcal{M}_1^*$, contains only isolated points, and $x_0$ is the only accumulation point of $\mathcal{M}_1^*$ (that is,
there is a sequence $\{x_i^I\}$ with $x_i^I\in\mathcal{M}_1^*,\,\forall i\geq1$
such that $x^I_i\to x_0\in(0,1)$ as $i\to\infty$) and $x_0$ is also a point of local minimum of $c$ (in the usual sense).
Then, one can appeal to the analysis of this paper to show
 $$
 \lim\limits_{s\to\infty}\lambda_1(s)=\min\{c(x_0),\,\inf\{c(x),\ x\in\mathcal{M}_1^*\}\}.
 $$
Note that $x_0$ may not be an isolated point of local maximum of such a given $m$. This implies that
such an oscillating behavior of $m$ can affect the limiting profile of the principal eigenvalue $\lambda_1(s)$.

\vskip4pt
Example B: Assume that the set of points of local maximum of $m$ is given by \eqref{max} in which we
now take $h_2=\infty,\,h_5=\infty$ so that $a_i^D<x_0<b_i^I$ for all $i\geq1$, $a_i^I,\,b_i^I\to x_0\in(0,1)$ and
$a_i^D,\,b_i^D\to x_0$ as $i\to\infty$. Then the limit $ \lim\limits_{s\to\infty}\lambda_1(s)$ is given as in Theorem \ref{th1.2}
with $h_2,\,h_5$ replaced by $\infty$ and so $\lim\limits_{s\to\infty}\lambda_1(s)$ is independent of such $x_0$. In other words,
such an oscillating behavior of $m$ has no qualitative effect on the limiting profile of the principal eigenvalue $\lambda_1(s)$.

Thus, it would be interesting to discuss the asymptotic behavior of the principal eigenvalue
when the advection function $m$ allows general oscillation.

The outline of this paper is as follows. Section 2 is devoted to the analysis of the limiting behavior of the principal eigenvalue
and Theorems \ref{th1.2} and \ref{th1.2a} are proved, while section 3 concerns the asymptotic profile of the principal eigenfunction in which
Theorem \ref{th1.3} is verified. Throughout the paper, we use $|E|$ to stand for the Lebesgue measure of a given set $E\subset\bR$.

\section{The principal eigenvalue:\ Proof of Theorems \ref{th1.2} and \ref{th1.2a}}

This section aims to analyze the asymptotic behavior of the principal eigenvalue
of \eqref{p} and \eqref{1.p1} as $s\to\infty$. We first consider the Neumann problem \eqref{1.1}, and
then investigate the general problem \eqref{p} and the periodic problem \eqref{1.p1}.

\subsection{The principal eigenvalue problem \eqref{1.1}} First of all, we estimate the upper bounds of $\lambda_1^\mathcal{N}(s)$, and state that

\begin{lem}\label{l2.1} Let $\lambda^*=\limsup\limits_{s\to\infty}\lambda_1^\mathcal{N}(s)$. The following assertions hold.
 \bes
 \left.\begin{array}{ll}
 \medskip
 \mbox{\rm{(i)}\ \  If $x^{I}\in\mathcal{M}_1$, then $\lambda^*\leq c(x^{I});$}\  & \ \mbox{\rm{(ii)}\ \  If $[a^I,b^I]\in{M}_2$, then $\lambda^*\leq\lambda_1^{\mathcal{ND}}(a^I,b^I);$}\\
 \displaystyle
 \mbox{\rm{(iii)}\ \  If $[a^I,b^D]\in\mathcal{M}_3$, then $\lambda^*\leq\lambda_1^{\mathcal{NN}}(a^I,b^D);$}\  & \ \mbox{\rm{(iv)}\ \  If $[a^D,b^I]\in\mathcal{M}_4$, then $\lambda^*\leq\lambda_1^{\mathcal{DD}}(a^D,b^I);$}\\
 \mbox{\rm{(v)}\ \  If $[a^D,b^D]\in\mathcal{M}_5$, then $\lambda^*\leq\lambda_1^{\mathcal{DN}}(a^D,b^D);$}\  & \ \mbox{\rm{(vi)}\ \  If $[0,a^I]\in\mathcal{M}_6$, then $\lambda^*\leq\lambda_1^{\mathcal{ND}}(0,a^I);$}\\
 \mbox{\rm{(vii)}\ \  If $[0,a^D]\in\mathcal{M}_7$, then $\lambda^*\leq\lambda_1^{\mathcal{NN}}(0,a^D);$}\ &\  \mbox{\rm{(viii)}\ \  If $[a^I,1]\in\mathcal{M}_8$, then $\lambda^*\leq\lambda_1^{\mathcal{NN}}(a^I,1);$}\\
 \mbox{\rm{(ix)}\ \  If $[a^D,1]\in\mathcal{M}_9$, then $\lambda^*\leq\lambda_1^{\mathcal{DN}}(a^D,1).$}
 \end{array}
 \right.
 \nonumber
 \ees

\end{lem}

\begin{proof} The assertion (i) follows directly from Lemma 2.2 of \cite{CL1}. In  the sequel,
we are going to verify the assertion (iii) by modifying the argument of Lemma 2.2 of \cite{CL1}.
Without loss of generality, we may assume that $c(x)\geq0$ on $[0,1]$; otherwise, we replace
$(\lambda_1^\mathcal{N}(s),\,c(x))$ in \eqref{1.1} by $(\lambda_1^\mathcal{N}(s)+\max_{[0,1]}|c(x)|,\,c(x)+\max_{[0,1]}|c(x)|)$.

According to our definition for $[a^I,b^D]$, there exists a small $\epsilon_0>0$ such that $m$ is constant on $[a^I,b^D]$,
non-decreasing on $[a^I-\epsilon_0,a^I]$, and non-increasing on $[b^{{D}},b^{{D}}+\epsilon_0]$. Furthermore,
for any small $\epsilon>0$, $\{x\in[0,1]:\ m'(x)>0\}\cap(a^{{I}}-\epsilon,a^{{I}})\not=\emptyset$ and
$\{x\in[0,1]:\ m'(x)<0\}\cap(b^{{D}},b^{{D}}+\epsilon)\not=\emptyset$.
We take $\varphi_0$ to be an principal eigenfunction corresponding to $\lambda_1^{\mathcal{NN}}(a^I,b^D)$.
Then, for any given three positive constant sequences $\{\alpha_i\}_{i=1}^\infty,\,\{\beta_i\}_{i=1}^\infty,\,\{\gamma_i\}_{i=1}^\infty$
satisfying $0<\alpha_i<\beta_i<\gamma_i<\epsilon_0$ for each $i\geq1$, and $\gamma_i\to0$ (and so $\alpha_i,\,\beta_i\to0$) as $i\to\infty$,
we may assume, without loss of generality, that $m(a^I-\beta_i)-m(a^I-\alpha_i)<0$ and $m(b^D+\beta_i)-m(b^D+\alpha_i)<0$ for each $i\geq1$.

We now choose the continuous function sequence $\{u_i\}_{i=1}^\infty$:
 \bes
 u_i(x)=\left\{\begin{array}{cl}
 0 &\ \ \ \mbox{if $x\in[0,a^I-\gamma_i]$},\\
 {{\gamma_i-a^I+x}\over{\gamma_i-\beta_i}}\varphi_0(a^I) &\ \ \ \mbox{if $x\in(a^I-\gamma_i,a^I-\beta_i]$},\\
 \varphi_0(a^I) &\ \ \ \mbox{if $x\in(a^I-\beta_i,a^I]$},\\
 \varphi_0(x) &\ \ \ \mbox{if $x\in(a^I,b^D]$},\\
 \varphi_0(b^D) &\ \ \ \mbox{if $x\in(b^D,b^D+\beta_i]$},\\
 {{\gamma_i+b^D-x}\over{\gamma_i-\beta_i}}\varphi_0(b^D) &\ \ \ \mbox{if $x\in(b^D+\beta_i,b^D+\gamma_i]$},\\
 0&\ \ \ \mbox{if $x\in(b^D+\gamma_i,1]$}.
 \end{array}
 \right.
 \nonumber
 \ees
By the variational characterization for  $\lambda_1(s)$ and $\lambda_1^{\mathcal{NN}}(a^I,b^D)$, elementary computation gives
 \bes
 \left.\begin{array}{lll}
 \medskip
 \displaystyle
 \lambda_1^\mathcal{N}(s)&\leq&
 \medskip
 \displaystyle
 {{\int_0^1 e^{2sm(x)}[(u'_i)^2+c(x)u^2_i]dx}\over{\int_0^1 e^{2sm(x)}u^2_idx}}\\
 \medskip
 &\leq&
 \displaystyle
 {{\int_{a^I}^{b^D} e^{2sm(x)}[(u'_i)^2+c(x)u^2_i]dx}\over{\int_{a^I}^{b^D}e^{2sm(x)}u^2_idx}}
 +{{\Big(\int_{a^I-\gamma_i}^{a^I}+\int_{b^D}^{b^D+\gamma_i}\Big)\Big\{e^{2sm(x)}[(u'_i)^2+c(x)u^2_i]\Big\}dx}\over{\int_0^1 e^{2sm(x)}u^2_idx}}\\
 \medskip
 &\leq&
 \displaystyle
 \lambda_1^{\mathcal{NN}}(a^I,b^D)+I+II,
 \end{array}
 \right.
 \nonumber
 \ees
where
 $$
 I={{\int_{a^I-\gamma_i}^{a^I} e^{2sm(x)}[(u'_i)^2+c(x)u^2_i]]dx}\over{\int_0^1 e^{2sm(x)}u^2_idx}},\ \
 II={{\int_{b^D}^{b^D+\gamma_i} e^{2sm(x)}[(u'_i)^2+c(x)u^2_i]dx}\over{\int_0^1 e^{2sm(x)}u^2_idx}}.
 $$
Since $m(x)$ is non-decreasing in $[a^I-\gamma_i,a^I]$ for each $i\geq1$ and $m(x)$ is constant on $[a^I,b^D]$, for any $s\geq0$ we deduce
 \bes
 \left.\begin{array}{lll}
 \medskip
 \displaystyle
 I&\leq&
 \medskip
 \displaystyle
 {\max_{[0,1]}|c(x)|{\int_{a^I-\gamma_i}^{a^I} e^{2sm(x)}u^2_idx}\over{\int_{a^I}^{b^D}e^{2sm(x)}u^2_idx}}
 +{{\int_{a^I-\gamma_i}^{a^I-\beta_i} e^{2sm(x)}(u'_i)^2dx}\over{\int_{a^I-\alpha_i}^{a^I} e^{2sm(x)}u^2_idx}}\\
 \medskip
 &\leq&
 \displaystyle
 {\max_{[0,1]}|c(x)|{\int_{a^I-\gamma_i}^{a^I} e^{2sm(a^I)}u^2_idx}\over{\int_{a^I}^{b^D}e^{2sm(a^I)}u^2_idx}}
 +{{e^{2s[m(a^I-\beta_i)-m(a^I-\alpha_i)]}}\over{\alpha_i(\gamma_i-\beta_i)}}.
 \end{array}
 \right.
 \nonumber
 \ees
Due to $m(a^I-\beta_i)-m(a^I-\alpha_i)<0$ for each $i\geq1$, by sending $s\to\infty$ first and then
sending $i\to\infty$, we easily see that $I\to0$.
Similarly, the term $II\to0$ by sending $s\to\infty$ first and then sending $i\to\infty$.
Consequently, $\lambda^*\leq\lambda_1^{\mathcal{NN}}(a^I,b^D)$, as wanted.

The remaining assertions can be proved in a similar way, and the details are omitted.

\end{proof}

In order to estimate the lower bounds of $\lambda_1^\mathcal{N}(s)$, we need several key lemmas as follows.
In what follows, let us remember that $w(s,\cdot)$ is the positive solution of \eqref{1.3} corresponding to the principal eigenvalue
$\lambda_1^\mathcal{N}(s)$ with the normalization $\int_0^1w^2(s,x)dx=1$ for each $s>0$.
Let $\mu$ be the probability measure associated with $w(s_j,\cdot)$ defined through \eqref{1.7}.
Then we have

\begin{lem}\label{l2.5} The following assertions hold.

\begin{itemize}

  \item[(i)] Assume that there exists a closed interval $[a,b]\subset[0,1]$ such that $\lim_{j\to\infty}\int_a^b w^2(s_j,x)dx\to0$, then
 $\mu((a,b))=0$.

 \item[(ii)] Assume that there exists a closed interval $[a,b]\subset[0,1]$
such that $\mu([a,b])=0$, then $\lim_{j\to\infty}\int_a^b w^2(s_j,x)dx\to0$.

 \item[(iii)]  Assume that there exist two closed intervals $[a_*,a]$ and $[b,b_*]$ of $[0,1]$ with $0<a<b<1$
such that $\mu([a_*,a])=\mu([b,b_*])=0$, then $\lim_{j\to\infty}\int_a^b w^2(s_j,x)dx\to\mu((a,b))$.

\end{itemize}

\end{lem}

\begin{proof} The assertion (i) is trivial. Indeed, it is known that $\mu((a,b))=\sup\{\mu([\tilde a,\tilde b]):\ [\tilde a,\tilde b]\subset(a,b)\}$.
Also, using \eqref{1.7}, it easily follows $\mu([\tilde a,\tilde b])=0$ for any closed interval $[\tilde a,\tilde b]\subset(a,b)$. Thus, (i) holds.

We next prove (ii). Arguing indirectly, we suppose that there exist a constant $\epsilon_0>0$ and a subsequence of $\{s_j\}$,
denoted by itself for convenience, such that
 \bes
 \int_a^b w^2(s_j,x)dx\geq\epsilon_0,\ \ \forall j\geq1.
 \label{l2.5-a}
 \ees
By restricting $w^2(s_j,\cdot)$ to $(a,b)$, we may assume, up to a further subsequence, that
 \bes
 \lim_{j\to\infty}\int_a^bw^2(s_j,x)\zeta(x)dx=\int_{[a,b]}\zeta(x)\mu^*(dx),\ \ \forall \zeta\in C([a,b])
 \label{l2.5-b}
 \ees
for a unique Radon measure $\mu^*$. Taking $\zeta=1$ in \eqref{l2.5-b} and using \eqref{l2.5-a}, we have
$\mu^*([a,b])\geq\epsilon_0$.

On the other hand, for each $j\geq1$, we decompose $w^2(s_j,\cdot)$ as
 \bes
 w^2(s_j,\cdot)=w^2(s_j,\cdot)\chi_{_{[0,a)}}+w^2(s_j,\cdot)\chi_{_{[a,b]}}+w^2(s_j,\cdot)\chi_{_{(b,1]}},
 \label{l2.5-d}
 \ees
where $\chi_{_X}$ represents the characteristic function over a set $X\subset[0,1]$. As above, we assume that
 \bes
 \int_0^1w^2(s_j,x)\chi_{_{[0,a)}}\zeta(x)dx\to\int_{[0,1]}\zeta(x)\mu_1(dx),\ \
 \int_0^1w^2(s_j,x)\chi_{_{[a,b]}}\zeta(x)dx\to\int_{[0,1]}\zeta(x)\mu_2(dx),\
 \nonumber
 \ees
and
 $$
 \int_0^1w^2(s_j,x)\chi_{_{(b,1]}}\zeta(x)dx\to\int_{[0,1]}\zeta(x)\mu_3(dx)
 $$
for any $\zeta\in C([0,1])$, as $j\to\infty$, where $\mu_i\,(i=1,2,3)$ are the certain Radon measures.
In view of \eqref{l2.5-d}, it is easily seen from the definitions of $\mu,\mu_i\,(i=1,2,3)$ that $\mu=\mu_1+\mu_2+\mu_3$ on $[0,1]$.

We further note that $\mu_2=\mu^*$ on $[a,b]$. Indeed, since $w^2(s_j,\cdot)\chi_{_{[a,b]}}=0$ on $[0,a)\cup(b,1]$,
the analysis similar to that of the assertion (i) concludes that
$\mu_2([0,a))=\mu_2((b,1])=0$. In addition,
 $$
 \int_0^1w^2(s_j,x)\chi_{_{[a,b]}}\zeta(x)dx=\int_a^bw^2(s_j,x)\zeta(x)dx,\ \ \forall \zeta\in C([a,b]).
 $$
Hereafter, we extend $\zeta$, which is defined on $[a,b]$, continuously to $[0,1]$ so that the integral over $[0,1]$ makes sense.
So it follows
 \bes
 \left.\begin{array}{lll}
 \medskip
 \displaystyle
 \int_{[a,b]}\zeta(x)\mu^*(dx)&=&
 \medskip
 \displaystyle
 \int_{[0,1]}\zeta(x)\mu_2(dx)=\int_{[0,a)}\zeta(x)\mu_2(dx)
 +\int_{[a,b]}\zeta(x)\mu_2(dx)\\
 &&\ +
 \displaystyle
 \int_{(b,1]}\zeta(x)\mu_2(dx)=\int_{[a,b]}\zeta(x)\mu_2(dx),\ \ \ \ \forall \zeta\in C([a,b]),
 \end{array}
 \right.
 \nonumber
 \ees
which obviously implies $\mu_2=\mu^*$ on $[a,b]$.

Therefore, we obtain that
 $$
 \mu([a,b])=\mu_1([a,b])+\mu_2([a,b])+\mu_3([a,b])=\mu_1([a,b])+\mu^*([a,b])+\mu_3([a,b])\geq\mu^*([a,b])\geq\epsilon_0,
 $$
which arrives at a contradiction with our assumption $\mu([a,b])=0$. Thus, (ii) is proved.

Lastly, we verify (iii). In \eqref{1.7}, we choose $\zeta=1$ on $[a,b]$, $\zeta=0$ on $[0,a_*]\cup[b_*,1]$,
and $0\leq\zeta\leq1$ in $[a_*,a]\cup[b,b_*]$ so that $\zeta\in C([0,1])$. Thanks to $\mu([a_*,a])=\mu([b,b_*])=0$,
it then follows from the assertion (ii) that
 \bes
 \lim_{j\to\infty}\int_0^1w^2(s_j,x)\zeta(x)dx=\lim_{j\to\infty}\int_a^bw^2(s_j,x)dx.
 \nonumber
 \ees
On the other hand, for such chosen $\zeta$, because of $\mu([a_*,a])=\mu([b,b_*])=0$, we have
 $$
 \int_{[0,1]}\zeta(x)\mu(dx)=\int_{(a,b)}\zeta(x)\mu(dx)=\int_{(a,b)}\mu(dx)=\mu((a,b)).
 $$
In light of \eqref{1.7}, the desired conclusion is established.

\end{proof}

\begin{lem}\label{l2.6} The following assertions hold.

\begin{itemize}

 \item[(i)] Given any $x\in(0,1)$ and any $\epsilon$ with $0<\epsilon<{1\over 2}\min\{x,\,1-x\}$, then
 $$
 \lim_{j\to\infty}\int_{x-2\epsilon}^{x+2\epsilon}w^2(s_j,x)c(x)dx\geq\min_{[x-\epsilon,x+\epsilon]}c(x)\, \mu(\{x\}).
 $$

 \item[(ii)] Given any $\epsilon$ with $0<\epsilon<{1/2}$, then
 $$
 \lim_{j\to\infty}\int_0^{2\epsilon}w^2(s_j,x)c(x)dx\geq\min_{[0,\epsilon]}c(x)\, \mu(\{0\}).
 $$

 \item[(iii)]Given any $\epsilon$ with $0<\epsilon<{1/2}$ such that
 $$
 \lim_{j\to\infty}\int_{1-2\epsilon}^1w^2(s_j,x)c(x)dx\geq\min_{[1-\epsilon,1]}c(x)\, \mu(\{1\}).
 $$

\end{itemize}

\end{lem}

\begin{proof} This lemma is a straightforward consequence of \eqref{1.7} and \eqref{1.8}.

\end{proof}

The following preliminary results aim to give a precise description of the support of $\mu$.

\begin{lem}\label{l2.2} Denote $\Omega_1=\{x\in(0,1):\ |m'(x)|>0\}\cup\{0:\ m'(0)>0\}\cup\{1:\ m'(1)<0\}$.
Then $\mu(\Omega_1)=0$.

\end{lem}

\begin{proof} This result follows directly from Lemma 3.1 and Lemma 3.4 of \cite{CL1}.

\end{proof}

\begin{lem}\label{l2.3} The following assertions hold.

 \begin{itemize}

 \item[(i)] Assume that $|m'(x)|>0$ in an open interval $(a,b)\subset[0,1]$.
Then $w(s,\cdot)\to0$ locally uniformly in $(a,b)$ as $s\to\infty;$

 \item[(ii)] Assume that $m'(x)>0$ in an interval $[0,a)\subset[0,1]$.
Then $w(s,\cdot)\to0$ locally uniformly in $[0,a)$ as $s\to\infty$.

 \item[(iii)] Assume that $m'(x)<0$ in an interval $(a,1]\subset[0,1]$.
Then $w(s,\cdot)\to0$ locally uniformly in $(a,1]$ as $s\to\infty$.

 \end{itemize}

\end{lem}

\begin{proof} Under the assumption of (i), we first claim that, by passing to a sequence, $w(s_j,\cdot)\to0$ a.e. in $(a,b)$ as $j\to\infty$.
In fact, given small $\delta>0$, we take $\zeta\in C([0,1])$ in \eqref{1.7} such that
$\zeta(x)=1$ on $[a+2\delta,b-2\delta]$, $\zeta(x)=0$ on $[0,a+\delta]\cup[b-\delta,1]$ and $0\leq\zeta(x)\leq1$ on $[0,1]$. Then
from \eqref{1.7} and Lemma \ref{l2.2} it follows that
 \bes
 \left.\begin{array}{lll}
 \medskip
 \displaystyle
 0\leq\lim_{j\to\infty}\int_{a+2\delta}^{b-2\delta}w^2(s_j,x)dx&\leq&
 \medskip
 \displaystyle
 \lim_{j\to\infty}\int_{a+\delta}^{b-\delta}w^2(s_j,x)\zeta(x)dx\\
 &=&
 \displaystyle
 \int_{[a+\delta,b-\delta]}\zeta(x)\mu(dx)\leq\mu([a+\delta,b-\delta])=0.
 \end{array}
 \right.
 \nonumber
 \ees
Clearly, this implies that $w(s_j,\cdot)\to0$ a.e. in $(a+2\delta,b-2\delta)$ as $j\to\infty$. Since $\delta$ can be arbitrarily small,
we have $w(s_j,\cdot)\to0$ a.e. in $(a,b)$, as claimed.

Given small $\delta>0$, by our assumption, there exists a positive constant $c_0=c_0(\delta)<1$ such that
$|m'(x)|\geq c_0$ and $|m''(x)|\leq 1/{c_0}$ on $[a+\delta,b-\delta]$.  Thanks to the claim proved before, we assume, without loss of generality, that
$w(s_j,a+\delta)\to0$ and $w(s_j,b-\delta)\to0$ as $j\to\infty$. Thus, we can find a large integer $J_0$ such that
$w(s_j,a+\delta),\,w(s_j,b-\delta)\leq1,\,\forall j\geq J_0$.
Furthermore, taking larger $J_0$ if necessary and using \eqref{1.3}, we observe that, for all $j\geq J_0$, $w(s_j,x)$ satisfies
 $$
 -w''(s_j,x)\leq-{1\over2}c_0^2s_j^2w(s_j,x),\ \ a+\delta<x<b-\delta.
 $$
Hence, for each $j\geq J_0$, $w(s_j,x)$ is a subsolution of the elliptic problem:
 \bes
 -u''=-{1\over2}c_0^2s_j^2u,\ \ a+\delta<x<b-\delta;\ \ \ u(s_j,a+\delta)=u(s_j,b-\delta)=1.
 \label{2.3}
 \ees
Simple analysis shows that the unique solution of \eqref{2.3}, denoted by $u_j$, satisfies $
w(s_j,\cdot)\leq u_j\to0$ locally uniformly in $(a+\delta,b-\delta)$ as $j\to\infty$. Due to the arbitrariness of $\delta$,
$w(s_j,\cdot)\to0$ locally uniformly in $(a,b)$, and so the assertion (i) holds.

The assertions (ii) and (iii) can be proved similarly. Indeed, to verify (ii), since $m'(0),\,w(s_j,0)>0$,
we get from the boundary condition in \eqref{1.3} that $w'(s_j,0)>0$ for each $j\geq1$.
As above, for any given small $\delta>0$, there exist a positive constant $c_0<1$ and a large integer $J_0$ such that, for all $j\geq J_0$, $w(s_j,x)$ satisfies
 $$
 -w''(s_j,x)\leq-{1\over2}c_0^2s_j^2w(s_j,x),\ \ 0<x<a-\delta; \ \ \ w'(s_j,0)>0,\ \ w(s_j,a-\delta)\leq1.
 $$
On the other hand, instead of considering the auxiliary problem \eqref{2.3}, we resort to the problem:
 \bes
 -u''=-{1\over2}c_0^2s_j^2u,\ \ 0<x<a-\delta;\ \ \ u'(s_j,0)=0,\ \ u(s_j,a-\delta)=1.
 \label{2.3a}
 \ees
It is easy to check that the unique solution $u_j$ of \eqref{2.3a} converges to $0$ locally uniformly in $[0,a-\delta)$.
Clearly, $w(s_j,x)$ is a subsolution of \eqref{2.3a}. A simple comparison argument asserts $w(s_j,\cdot)\leq u_j$, and so
$w(s_j,\cdot)\to0$ locally uniformly in $[0,a)$ as $j\to\infty$. Thus, Lemma \ref{l2.3} is proved.
\end{proof}

\begin{lem}\label{l2.4} Assume that $m$ is strictly increasing or strictly decreasing on $[a_1,a_2]\subset[0,1]$, then $\mu((a_1,a_2))=0$.

\end{lem}

\begin{proof} We only consider the case that $m$ is strictly increasing on $[a_1,a_2]$
since the assertion in the other case can be established similarly.
To the end, it suffices to show that $\mu((a_1,b_i])=0,\,\forall i\geq1$ for a given sequence $\{b_i\}_{i\geq1}$
satisfying $b_i<a_2$ for each $i\geq1$ and $b_i\to a_2$ as $i\to\infty$.

For any fixed $i\geq1$, we first claim that
 \bes
 \label{claim1}
 \mbox{$\|w(s_j,\cdot)\|_{L^\infty(a_1,b_i)}\leq M$ for some positive constant $M$, independent of $j\geq1$.}
 \ees
We proceed by an indirect argument. To produce a contradiction, the analysis below turns out to be rather long,
and for clarity we divide it into three steps.

\vskip6pt
{\it Step 1.} Suppose that the claim \eqref{claim1} is false. Then, there is a sequence of $\{s_j\}_{j\geq1}$,
labelled by itself for convenience, such that $\|w(s_j,\cdot)\|_{L^\infty(a_1,b_i)}\to\infty$ as $j\to\infty$.
For each $j$, we can find $x_j\in[a_1,b_i]$ such that $w(s_j,x_j)\to\infty$ as $j\to\infty$. By taking
 $$
 M=2\Big(1+\max_{[0,1]}|c(x)|\Big)^{1/2},
 $$
we have $w(s_j,x_j)\geq M,\,\forall j\geq j_0$ for some large integer $j_0$.

On the other hand, as $m$ is strictly increasing on $[a_1,a_2]$, it is clear that
$\{m'(x)>0\}\cap(b_i,a_2)\not=\emptyset$. So one can find a $y_0\in(b_i,a_2)$ such that
$m'(x)>0$ in a small neighbourhood of $y_0$. Hence, Lemma \ref{l2.3} ensures that,
by taking larger $j_0$ if necessary, $w(s_j,y_0)\leq 1/M$ for all $j\geq k_0$.

From now on, we fix $j=j_0$. For simplicity, denote
 $$
 \mbox{$a_*=x{_{j_0}},\ b_*=y_0$\ \ and\ \ $w(x)=w(s_{_{j_0}},x)$.}
 $$
Thus, $a_*<b_*$, and
 \bes
 \label{2.s1}
 w(b_*)\leq 1/M<M\leq w(a_*).
 \ees
In view of $w\in C^1([a_*,b_*])$ and \eqref{2.s1}, clearly, $\{x\in(a_*,b_*): \ w'(x)<0\}$
is a nonempty open set, which is therefore the union of at most countably many disjoint open intervals
and
 $$
 \mathcal{G}:=\{x\in[a_*,b_*]:\ w'(x)=0\}
 $$
is a closed set. So we assume that
 $$
 \{x\in(a_*,b_*): \ w'(x)<0\}=\bigcup_{i\in\mathbb{N}}(\hat a_i,\hat b_i),
 $$
where $\mathbb{N}$ is a given set consisting of at most countably many integers, and
$(\hat a_i,\hat b_i)\cap(\hat a_j,\hat b_j)=\emptyset,\,\forall i\not=j$. In addition, since
$w$ is strictly decreasing in each $(\hat a_i,\hat b_i)$, it is easily seen that $(w(\hat b_i),w(\hat a_i))$ is an open interval and
 \bes
 \label{2.s2}
 (w(b_*),w(a_*))\subseteq w(\mathcal{G})\cup \bigcup_{i\in\mathbb{N}}(w(\hat b_i),w(\hat a_i)).
 \ees

From Sard's Lemma (see, for instance, Theorem 3.6.3 of \cite{AMR}) it follows that the Lebesgue measure $|w(\mathcal{G})|=0$.
This fact, together with \eqref{2.s2}, allows us to assert that
 \bes
 \label{2.s3}
 \Big|\bigcup_{i\in\mathbb{N}}(w(\hat b_i),w(\hat a_i))\Big|=\Big|w(\mathcal{G})\cup \bigcup_{i\in\mathbb{N}}(w(\hat b_i),w(\hat a_i))\Big|\geq w(a_*)-w(b_*).
 \ees

\vskip6pt
{\it Step 2.} The rearrangement of the curve sequence $\{(x,w(x)):\ x\in(\hat a_i,\hat b_i)\}_{i\in\mathbb{N}}$. We will proceed in three substeps. Set
 $$
 \mathcal{F}=\{(\hat a_i,\hat b_i):\ i\in\mathbb{N}\}.
 $$
Recall that $w$ is strictly decreasing in each $(\hat a_i,\hat b_i)\in\mathcal{F}$. We take an arbitrary interval, say, $(\hat a_1,\hat b_1)\in\mathcal{F}$.

As a first substep, we are going to do the upward extension for the $C^1$-curve $\{(x,w(x)):\,x\in[\hat a_1,\hat b_1]\}$
by picking up some other elements from the curve sequence $\{(x,w(x)):\ x\in(\hat a_i,\hat b_i)\}_{i\in\mathbb{N}}$
appropriately. We start at the highest point $(\hat a_1,w(\hat a_1))$ of $\{(x,w(x)):\,x\in[\hat a_1,\hat b_1]\}$, and operate in the following procedures.

If $w(\hat a_1)\geq w(x)$ for all $x\in[\hat a_i,\hat b_i]$ and any $i\in\mathbb{N}$, then we do not need conduct
the upward extension, and just define $\underline w_1(x)=w(x)$ on $[\hat a_1,\hat b_1]$.

If there exists some $(\underline a,\underline b)\in\mathcal{F}$ satisfying $w(\hat a_1)\in[w(\underline b),w(\underline a))$,
then there is a unique $\underline c\in(\underline a,\underline b]$ such that $w(\hat a_1)=w(\underline c)<w(x)$
for all $x\in[\underline a,\underline c)$. Then, we move the curve $\{(x,w(x)):\,x\in[\underline a,\underline c]\}$
by horizontal translation so that its lowest point $(\underline c,w(\underline c))$ overlaps the highest point
$(\hat a_1,w(\hat a_1))$ of the curve $\{(x,w(x)):\,x\in[\hat a_1,\hat b_1]\}$.
Thus, we obtain an extended curve, denoted by $\{(x,\underline w_1(x)):\,x\in[\underline a_1,\hat b_1]$, with the
highest point $(\underline a_1, w(\underline a))$, where $\underline a_1=\hat a_1+\underline a-\underline c$ and
$\underline w_1$ is a continuous and strictly decreasing function on $[\underline a_1,\hat b_1]$.

Now, if there exists some $(\tilde a,\tilde b)\in\mathcal{F}$ satisfying $w_1(\underline a_1)\in[w(\tilde b),w(\tilde a))$,
then there is a unique $\tilde c\in(\tilde a,\tilde b]$ such that $w_1(\underline a_1)=w(\tilde c)<w(x)$ for all
$x\in[\tilde a,\tilde c)$. Similarly as before, translating the curve $\{(x,w(x)):\,x\in[\tilde a,\tilde c]\}$ until its
lowest point $(\tilde c,w(\tilde c))$ coincides with the highest point $(\underline a_1,w_1(\underline a_1))$ of the curve
$\{(x,w_1(x)):\,x\in[\underline a_1,\hat b_1]\}$, we obtain an extended continuous curve which has the highest point
$(\underline a_2, w(\tilde a))$, where $\underline a_2=\underline a_1+\tilde a-\tilde c$. We use $\underline w_2$,
which is defined on $[\underline a_2,\hat b_1]$ and continuous and strictly decreasing, to represent the function of such a new curve.

If such $(\tilde a,\tilde b)\in\mathcal{F}$ can not be found, we then define $\underline w_2(x)=\underline w_1(x)$
on $[\underline a_1,\hat b_1]$.

In a similar way, we conduct the upward extension for the curve $\{(x,\underline w_2(x)):\,x\in[\underline a_2,\hat b_1]\}$.
After at most countably many times, we obtain the continuous curve $\{(x,\underline w_1(x)):\ x\in[a_{1,+},\hat b_1]\}$,
where $a_{1,+}=\inf\{\underline a_0,\underline a_1,\underline a_2,\cdots\}\geq \hat b_1-(b_*-a_*)$,
and $\underline W_1(x)=\underline w_k(x),\ \forall x\in[\underline a_k,\hat b_1],\ k=0,1,2,\cdots$, and
$\underline W_1(a_{1,+})\not\in[w(\hat b_i),w(\hat a_i)),\,\forall (\hat a_i,\hat b_i)\in\mathcal{F}$.
Here $\underline a_0=\hat a_1$. Note that a continuous extension from right was made here at $x=a_{1,+}$
for $\underline W_1$ so that $\underline W_1\in C([a_{1,+},\hat b_1]$ if the above procedures are of infinite times.
So far, we have finished the upward extension of the original curve $\{(x,w(x)):\,x\in[\hat a_1,\hat b_1]\}$.

\vskip 6pt

Secondly, we carry out the downward extension for $\{(x,\underline W_1(x)):\,x\in[a_{1,+},\hat b_1]\}$, starting at its lowest point $(\hat b_1,\underline W_1(\hat b_1))$. The procedures are similar to those of the upward extension:

If we can find some $(\overline a,\overline b)\in\mathcal{F}$ satisfying $\underline W_1(\hat b_1)\in(w(\overline b),w(\overline a)]$, then there is a unique $\overline c\in[\overline a,\overline b)$ such that $\underline W_1(\hat b_1)=w(\overline c)>w(x),\ \forall x\in(\overline c,\overline b]$. Again, horizontally one can translate the curve $\{(x,w(x)):\,x\in[\overline c,\overline b]\}$ until its highest point $(\overline c,w(\overline c))$ overlaps the lowest point $(\hat b_1,\underline W_1(\hat b_1))$ of $\{(x,\underline W_1(x)):\,x\in[a_{1,+},\hat b_1]\}$. Denote by $\{(x,\overline W_1(x)):\ x\in[a_{1,+},\overline b_1]\}$ such an extended curve. Then, $0<\overline b_1-\hat b_1=\overline b-\overline c$ and $\overline W_1$ is a strictly decreasing continuous function on $[a_{+,1},\overline b_1]$.

If such $(\overline a,\overline b)\in\mathcal{F}$ does not exist, we then define $\overline W_1(x)=\underline W_1(x)$ on $[a_{1,+},\hat b_1]$.

Set $\overline b_0=\hat b_1$. Proceeding similarly as above, we do the possible downward extension for the curve $\{(x,\overline W_1(x)):\,x\in[a_{1,+},\overline a_1]\}$. After at most countably many times, we can obtain the curve $\{(x,W_1(x)):\ x\in[a_{1,+},b_1^+]\}$, where $b_1^+=\sup\{\overline b_0,\overline b_1,\overline b_2,\cdots\}$ satisfying $0<b_1^+-a_{1,+}\leq b_*-a_*$, and $W_1\in C([a_{1,+},b_1^+])$ enjoys the properties: $W_1(x)=\overline W_k(x),\ \forall x\in[a_{1,+},\overline b_k],\ k=0,1,2,\cdots$, $W_1$ is strictly decreasing on $[a_{1,+},b_1^+]$, and
 $$
 W_1(a_{1,+})\not\in[w(\hat b_i),w(\hat a_i)),\ \ W_1(b_1^+)\not\in(w(\hat b_i),w(\hat a_i)],\ \ \ \forall (\hat a_i,\hat b_i)\in\mathcal{F}.
 $$
This finishes the maximal extension of the curve $\{(x,w(x)):\,x\in[\hat a_1,\hat b_1]\}$.

\vskip6pt Thirdly, if $W_1$ satisfies $W_1(a_{1,+})-W_1(b_1^+)\geq w(a_*)-w(b_*)$, we stop. If $W_1(a_{1,+})-W_1(b_1^+)<w(a_*)-w(b_*)$, by \eqref{2.s3} and the extension conducted above, it is easy to see that there is an interval, say, $(\hat a_2,\hat b_2)\in\mathcal{F}$ such that $[w(\hat b_2),w(\hat a_2)]\cap [W_1(b_1^+),W_1(a_{1,+})]=\emptyset$. In such a situation, following the same procedures as before, we extend the curve $\{(x,w(x)):\,x\in[\hat a_2,\hat b_2]\}$ to a maximal one, whose function $W_2$, defined on an interval, say $[a_{2,+},b_2^+]\supset[\hat a_2,\hat b_2]$, is continuous and strictly decreasing. Moreover, $\sum_{i=1}^2(b_i^+-a_{i,+})\leq b_*-a_*$, and
 $$
 W_2(a_{2,+})\not\in[w(\hat b_i),w(\hat a_i)),\ \ W_2(b_2^+)\not\in(w(\hat b_i),w(\hat a_i)],\ \ \ \forall (\hat a_i,\hat b_i)\in\mathcal{F}.
 $$

If
 $$
 W_1(a_{1,+})-W_1(b_1^+)+W_2(a_{2,+})-W_2(b_2^+)\geq w(a_*)-w(b_*),
 $$
we stop; if
 $$
 W_1(a_{1,+})-W_1(b_1^+)+W_2(a_{2,+})-W_2(b_2^+)<w(a_*)-w(b_*),
 $$
arguing as above, we can find an interval, say, $(\hat a_3,\hat b_3)\in\mathcal{F}$ such that
 $$
 [w(\hat b_3),w(\hat a_3)]\cap \cup_{i=1}^2([W_i(b_i^+),W_i(a_{i,+})])=\emptyset,
 $$
and then we extend the curve $\{(x,w(x)):\,x\in[\hat a_3,\hat b_3]\}$ to a maximal one, whose function $W_3$, defined on an interval,
say $[a_{3,+},b_3^+]\supset[\hat a_3,\hat b_3]$, is continuous and strictly decreasing, and
 $$
 \sum_{i=1}^3(b_i^+-a_{i,+})\leq b_*-a_*,\ W_3(a_{3,+})\not\in[w(\hat b_i),w(\hat a_i)),\ W_3(b_3^+)\not\in(w(\hat b_i),w(\hat a_i)],\ \forall (\hat a_i,\hat b_i)\in\mathcal{F}.
 $$

Up to at most countably many times, we obtain the function sequence $\{W_i:\ i\in\mathbb{E}\}$ with $\mathbb{E}$
being a given set consisting of at most countably many integers, which satisfies

 \begin{itemize}
 \item[(p1)] Each $W_i$ is continuous, and strictly decreasing on its domain $[a_{i,+},b_i^+]$ with
 $\sum_{i\in\mathbb{E}}(b_i^+-a_{i,+})\leq b_*-a_*$, and is $C^1([a_{i,+},b_i^+]\setminus\mathcal{O}_i)$ with each
 $\mathcal{O}_i$ containing at most countably many points;

 \item[(p2)] $(W_i(b_i^+),W_i(a_{i,+}))\cap(W_j(b_j^+),W_j(a_{j,+}))=\emptyset,\,\forall i\not=j$, and
 $W_j(a_{j,+})\not\in[w(\hat b_i),w(\hat a_i)),\ \, W_j(b_j^+)\not\in(w(\hat b_i),w(\hat a_i)],\ \, \forall j\in\mathbb{E},\, \forall (\hat a_i,\hat b_i)\in\mathcal{F}$;

 \item[(p3)] For any $i\in\mathbb{N}$ and any $x\in(\hat a_i,\hat b_i)\in\mathcal{F}$, there is a $W_j$ such that $w(x)\in [W_j(b_j^+),W_j(a_{j,+})]$;

 \item[(p4)] $\sum_{i\in\mathbb{E}}\int_{a_{i,+}}^{b_{i,+}}|W_i'(x)|^2dx\leq\sum_{i\in\mathbb{N}}\int_{\hat a_i}^{\hat b_i}|w'(x)|^2dx\leq\int_{a_*}^{b_*}|w'(x)|^2dx$.

 \end{itemize}
Hence, from (p3) and \eqref{2.s3}, it immediately follows that
 \bes
 \label{2.s4}
 \Big|\bigcup_{i\in\mathbb{E}}[W_i(b_i^+),W_i(a_{i,+})]\Big|\geq\Big|\bigcup_{i\in\mathbb{N}}(w(\hat b_i),w(\hat a_i))\Big|\geq w(a_*)-w(b_*).
 \ees

\vskip6pt
Finally, in view of (p1)-(p4), through at most countably many times of translation transformations
(including possible horizontal and vertical translations) over the curve sequence $\{(x,W_i(x)):\ x\in[a_{i,+},b_i^+],\ i\in\mathbb{E}\}$,
we can get a function $W$, which has the following properties:
 \begin{itemize}
 \item $W$ is strictly decreasing on its domain $[a_\infty,b_\infty]$ with $0<b_\infty-a_\infty\leq b_*-a_*$, is continuous at $a_\infty$ and $b_\infty$,
 and is $C^1([a_\infty,b_\infty]\setminus\mathcal{O})$ with $\mathcal{O}$ containing at most countably many points;

 \item $\Big|\bigcup_{i\in\mathbb{E}}[W_i(b_i^+),W_i(a_{i,+})]\Big|=W(a_\infty)-W(b_\infty)$, which, combined with \eqref{2.s4},
 implies that $W(a_\infty)-W(b_\infty)\geq w(a_*)-w(b_*)$;

 \item $\int_{a_\infty}^{b_\infty}|W'(x)|^2dx=\sum_{i\in\mathbb{E}}\int_{a_{i,+}}^{b_{i,+}}|W_i'(x)|^2dx
     \leq\sum_{i\in\mathbb{N}}\int_{\hat a_i}^{\hat b_i}|w'(x)|^2dx\leq\int_{a_*}^{b_*}|w'(x)|^2dx$.

 \end{itemize}

\vskip6pt {\it Step 3.} We use the same notation as in step 2. Based on what was proved by step 2,
together with the fact of $m'\geq0$ and $w'(x)\leq0$ on each $[\hat a_i,\hat b_i]\subset[0,1],\,\forall i\in\mathbb{N}$, we deduce from \eqref{1.4} that
 \bes
 \left.\begin{array}{lll}
 \medskip
 \displaystyle
 \max_{[0,1]}|c(x)|\geq\lambda^*&\geq&
 \medskip
 \displaystyle
 \int_{a_*}^{b_*}(w'(x)-sw(x)m'(x))^2dx-\max_{[0,1]}|c(x)|\\
 \medskip
 &\geq&
 \displaystyle
 \sum_{i\in\mathbb{N}}\int_{\hat a_i}^{\hat b_i}|w'(x)|^2dx-\max_{[0,1]}|c(x)|\geq
 \int_{a_\infty}^{b_\infty}|W'(x)|^2dx-\max_{[0,1]}|c(x)|.
 \end{array}
 \right.
 \label{2.6}
 \ees

We then aim to estimate the integral $\int_{a_\infty}^{b_\infty}|W'(x)|^2dx$. It is well known that $H^1((a_\infty,b_\infty))$ is compactly embedded into $C([a_\infty,b_\infty])$.
Let us consider the minimizer of the functional
 \bes
 \int_{a_\infty}^{b_\infty}|u'(x)|^2dx,\ \ \ \forall u\in H^1(a_\infty,b_\infty),
 \nonumber
 \ees
under the constrained conditions $u(a_\infty)=W(a_\infty),\, u(b_\infty)=W(b_\infty)$.
It is easy to check that the minimizer of such a functional is attainable and its minimal $u_0$ is a
solution of the following ODE problem with two-point boundary values
 \bes
 u_0''=0,\ \ \ x\in(a_\infty,b_\infty);\ \ \
 u_0(a_\infty)=W(a_\infty),\ \ u_0(b_\infty)=W(b_\infty).
 \nonumber
 \ees
Thus, $u_0$ is the segment connecting the two endpoints $a_\infty$ and $b_\infty$. So we have
 $$
 \int_{a_\infty}^{b_\infty}|W'(x)|^2dx\geq\int_{a_\infty}^{b_\infty}|u_0'(x)|^2dx={{|W(a_\infty)-W(b_\infty)|^2}\over{b_\infty-a_\infty}}.
 $$
Recall that $0<b_\infty-a_\infty\leq b_*-a_*\leq1$ and $W(a_\infty)-W(b_\infty)\geq w(a_*)-w(b_*)$. As a consequence,
this, together with \eqref{2.s1} and \eqref{2.6}, yields
 \bes
 2\max_{[0,1]}|c(x)|&\geq&
 \displaystyle
 {{|w(a_*)-w(b_*)|^2}\over{b_*-a_*}}\geq{1\over2}\Big(M-{1\over M}\Big)^2,
 \nonumber
 \ees
which leads to an obvious contraction due to the choice of $M$. Therefore, the claim \eqref{claim1} is proved.

As $m$ is strictly increasing on $[a_1,a_2]$, we know that
$m'(x)\geq0$ on $[a_1,a_2]$. Set $\mathcal{C}=\{x\in[a_1,a_2]:\ m'(x)=0\}$. In the sequel, we need consider two different cases:
Case {\bf A}:\ $|\mathcal{C}|=0$;  Case {\bf B}:\ $|\mathcal{C}|>0$.

We first treat Case {\bf A}:\ $|\mathcal{C}|=0$. For any small $\delta>0$, we take $\zeta(x)=1$
on $[a_1+\delta,b_i-\delta]$, $\zeta(x)=0$ on $[0,a_1+{1\over2}\delta]\cup[b_i-{1\over2}\delta,1]$ and $0\leq \zeta(x)\leq1$ on
$[a_1+{1\over2}\delta,a_1+\delta]\cup[b_i-\delta,b_i-{1\over2}\delta]$ so that $\zeta\in C([0,1])$. By Lemma \ref{l2.3}, we know that $w(s_j,\cdot)\to0$
a.e. in $\{m'(x)>0\}$. Hence, combined with this fact, $|\mathcal{C}|=0$ and the Lebesgue's dominated convergence theorem,
we can easily conclude from \eqref{1.7} and \eqref{claim1} that
 \bes
 \left.\begin{array}{l}
 \displaystyle
 \mu([a_1+\delta,b_i-\delta])\leq\lim\limits_{j\to\infty}\int_{(a_1+\delta/2,b_i-\delta/2)\cap\mathcal{C}}w^2(s_j,x)dx
 +\lim\limits_{j\to\infty}\int_{\{m'(x)>0\}}w^2(s_j,x)dx=0,
 \end{array}
 \right.
 \nonumber
 \ees
which therefore gives $\mu((a_1,b_i))=0,\,\forall i\geq1$ due to the arbitrariness of $\delta$, and in turn $\mu((a_1,a_2))=0$ as $b_i\to a_2$.

\vskip3pt We next consider Case {\bf B}:\ $|\mathcal{C}|>0$. Recall that $m$ is strictly increasing on $[a_1,a_2]$. Thus,
without loss of generality, we may assume that $m'(b_i)>0,\,\forall i\geq1$. In what follows, we fix $i\geq1$. Then $m'(x)>0$
on $(b_i-\epsilon_0/2,b_i)$ for some small $\epsilon_0>0$.

Set
 $$
 \mathcal{A}=\{x\in(0,1):\ m'(x)>0\}.
 $$
Since $\mathcal{A}$ is a bounded open set, it is a union of at most
countably many disjoint open intervals. So we can find a sequence of closed sets, say
$\{\mathcal{A}_k\}_{k=1}^{\infty}$ such that $\mathcal{A}_k\subset\mathcal{A}$,
$\mathcal{A}_k\subset\mathcal{A}_{k+1},\,k\geq 1$, $\bigcup_{k=1}^{\infty}\mathcal{A}_k=\mathcal{A}$.
Hence, given $k\geq 1$, for any small $\delta=\delta(k)>0$, there holds
 \bes
 \label{cl4}
 \mbox{$w(s_j,x)\leq\delta$ for all $x\in\mathcal{A}_k$ and for all large $j$ (due to Lemma \ref{l2.3}).}
 \ees

We now assert that for any given $0<\epsilon\leq\epsilon_0/2$, there is an integer $k_0$ such that
 \bes
 \label{claim2}
 d(x,\mathcal{A}_k\cap(x,b_i))<\epsilon,\ \ \forall x\in\{y\in[0,1]:\ m'(y)=0\}\cap(a_1,b_i),\ \ \mbox{for all}\ k\geq k_0.
 \ees
Here, $d(x,\mathcal{A}_k\cap(x,b_i))$ stands for the usual distance between the point $x$
and the set $\mathcal{A}_k\cap(x,b_i)$. Indeed, if \eqref{claim2} does not hold,
there is a subsequence $\{k_l\}_{l=1}^\infty$ with $k_l\to \infty$ as $l\to\infty$
and a point sequence $x_l\in\{y\in[0,1]:\ m'(y)=0\}\cap(a_1,b_i)$ such that
 \bes
 \nonumber
 \mbox{$d(x_l,\mathcal{A}_{_{k_l}}\cap(x_l,b_i))\geq\epsilon_1,\,\forall l\geq1$\ \ for some\ $0<\epsilon_1\leq\epsilon_0/2$}.
 \ees
If $\mathcal{A}_{_{k_l}}\cap(x_l,b_i)=\emptyset$ for some $l$, we define $d(x_l,\mathcal{A}_{_{k_l}}\cap(x_l,b_i))=\infty$.
As $x_l\in\{y\in[0,1]:\ m'(y)=0\}\cap(a_1,b_i)$ and $m'(x)>0$
on $(b_i-\epsilon_0/2,b_i)$, it is clear that $x_l\leq b_i-\epsilon_0$. Passing up to a subsequence,
we assume that $x_l\to x^*$ and so $x^*\leq b_i-\epsilon_0$.
Moreover, in view of the fact that $\mathcal{A}_{_{k_l}}\subset\mathcal{A}_{_{k_{l+1}}},\,\forall l\geq 1$ and
$\bigcup_{l=1}^{\infty}\mathcal{A}_{_{k_l}}=\mathcal{A}$,  by sending $l\to\infty$ it easily follows that
 $$
 d(x^*,\mathcal{A}\cap(x^*,b_i))\geq\epsilon_1,
 $$
which immediately implies that $m$ is constant on $[x^*,x^*+\epsilon_1]$,
an obvious contradiction! Hence, \eqref{claim2} holds.

We then conclude that
 \bes
 \label{claim3}
 \limsup_{j\to\infty}\|w(s_j,\cdot)\|_{L^\infty((a_1,b_i)\setminus\mathcal{A})}:=m^*=0.
 \ees
Once the assertion \eqref{claim3} holds, the argument in Case {\bf A} can be easily adapted to
show that $\mu((a_1,b_i))=0,\,\forall i\geq1$ and in turn $\mu((a_1,a_2))=0$.

It remains to prove \eqref{claim3}. Suppose that $m^*>0$. Then we can find a sequence $x_j\in(a_1,b_i)\cap\{m'(x)=0\})$
such that $w(s_j,x_j)\geq m^*/2,\,\forall j\geq j_0$ for some large $j_0$. On the other hand,
by taking $\delta=m^*/4$ and $\epsilon=\min\{{1\over 2}\epsilon_0,\ {{(m^*)^2}\over{64\max_{[0,1]}|c(x)|}}\}$
in \eqref{cl4} and \eqref{claim2}, respectively, we see that there exists $y_j$  with $0<y_j-x_j<\epsilon$ such that $w(s_j,y_j)<m^*/4$
for all $j\geq j_0$ by requiring $j_0$ to be larger if necessary.

Given $j\geq j_0$, similarly to the proof of the claim \eqref{claim1}, there exists a function $W$ satisfying
 \begin{itemize}
 \item $W$ is strictly decreasing on its domain $[a_\infty,b_\infty]$ with $0<b_\infty-a_\infty\leq y_j-x_j<\epsilon$, is continuous at $a_\infty$ and $b_\infty$,
 and is $C^1([a_\infty,b_\infty]\setminus\mathcal{O})$ with $\mathcal{O}$ containing at most countably many points;

 \item $W(b_\infty)-W(a_\infty)\geq w(s_j,x_j)-w(s_j,y_j)>m^*/4>0$;

 \item $\int_{a_\infty}^{b_\infty}|W'(x)|^2dx\leq\sum_{i\in\mathbb{N}}\int_{\hat a_i}^{\hat b_i}|w'(s_j,x)|^2dx$, where
$\bigcup_{i\in\mathbb{N}}(\hat a_i,\hat b_i)=\{x\in(x_j,y_j): \ w'(s_j,x)<0\}$, with $\mathbb{N}$ being a given
set consisting of at most countably many integers, $(\hat a_i,\hat b_i)\cap(\hat a_j,\hat b_j)=\emptyset,\,\forall i\not=j$.

 \end{itemize}
As $m'(x)\geq0$ on $[x_j,y_j]$ and $w'(s_j,x)\leq0$ on each $[\hat a_i,\hat b_i]$, the same analysis as in step 3 gives
 \bes
 \left.\begin{array}{lll}
 \medskip
 \displaystyle
 2\max_{[0,1]}|c(x)|&\geq&
 \medskip
 \displaystyle
 \int_0^1\{[w'(s_j,x)-s_jw(s_j,x)m'(x)]^2\}dx\\
 \medskip
 &\geq&
 \displaystyle
 \sum_{i\in\mathbb{N}}\int_{\hat a_i}^{\hat b_i}|w'(s_j,x)|^2dx\geq\int_{a_\infty}^{b_\infty}|W'(x)|^2dx\geq{{|W(b_\infty)-W(a_\infty)|^2}\over{b_\infty-a_\infty}}\\
 \medskip
 &\geq&
 \displaystyle
 {{|w(s_j,x_j)-w(s_j,y_j)|^2}\over{y_j-x_j}}\geq{{(m^*)^2}\over{16\epsilon}},
 \end{array}
 \right.
 \nonumber
 \ees
which is a contradiction because of the choice of $\epsilon$, and \eqref{claim3} is thus verified. So we have proved that $\mu((a_1,a_2))=0$.
This completes the proof of Lemma \ref{l2.4}.

\end{proof}

In fact, we can further show that in Lemma \ref{l2.4}, $\mu(\{a_1\})=0$ when $m$ is strictly increasing and $\mu(\{a_2\})=0$ when $m$ is strictly decreasing.
In order to cover other general cases, we next consider an interval $[a_1,a_2]\subset[0,1]$
on which $m$ is non-decreasing and $\{x\in[0,1]:\ m'(x)>0\}\cap(a_1,a_2)\not=\emptyset$. Without loss of generality, we assume that
$(a_1,a_1+\epsilon)\cap\{x\in[0,1]:\ m'(x)>0\}\not=\emptyset$ and $(a_2-\epsilon,a_2)\cap\{x\in[0,1]:\ m'(x)>0\}\not=\emptyset,\,\forall \epsilon>0$.
Denote by $\bigcup_{l=1}^{L}(c_l,d_l)$ the set of intervals where $m$ is constant on each $[c_l,d_l]\subset(a_1,a_2)$,
and $[c_l,d_l]\cap[c_j,d_j]=\emptyset,\,\forall l\not=j$,
where $L\geq0$ is a finite integer or equal to $\infty$. Here, we make a convention that $L=0$ means that $m$ is strictly increasing on $[a_1,a_2]$.
Then we are able to state

\begin{lem}\label{l2.4a} $\mu([a_1,a_2)\setminus\bigcup_{l=1}^{L}(c_l,d_l))=0$,
and if additionally $0<a_2<1$ and $m'(x)\geq0$ on $[a_2,a_2+\epsilon_0]$ for some small $\epsilon_0>0$, then
$\mu([a_1,a_2]\setminus\bigcup_{l=1}^{L}(c_l,d_l))=0$.

\end{lem}

\begin{proof} We first prove
 \bes
 \label{l2.4a-1}
 \mu((a_1,a_2)\setminus\bigcup_{l=1}^{L}[c_l,d_l])=0.
 \ees
When $L<\infty$, the proof is the same as in Lemma \ref{l2.4}. It remains to consider $L=\infty$.
According to our assumption, there exists a sequence $\{z_i\}_{i\geq1}$
with $a_1<z_i<a_2$ and $z_i\to a_2$ as $i\to\infty$ such that $m'(z_i)>0$ for each $i\geq1$. Thus, it is sufficient to show that
 \bes
 \label{l2.4a-2}
 \mu((a_1,z_i)\setminus\bigcup_{l=1}^{\infty}[c_l,d_l])=0\ \  \mbox{for any given $i\geq1$.}
 \ees
So from now on, we always fix $i$.

Note that $\{x\in[0,1]:\ m'(x)>0\}\cap(z_i,a_2)\not=\emptyset$. Then Lemma \ref{l2.3} implies that $w(s_j,\hat x)\to0$ for some $\hat x\in(z_i,a_2)$.
In view of this fact, one can appeal to the similar argument as in proving \eqref{claim1} to conclude that
 \bes
 \mbox{$\|w(s_j,\cdot)\|_{L^\infty(a_1,z_i)}\leq M$ for some positive constant $M$, independent of $j\geq1$.}
 \label{claim5}
 \ees

We use the same notation $\mathcal{A},\,\mathcal{A}_k,\,k\geq1$ as in the proof of Lemma \ref{l2.4}.
Given $\delta=\delta(k)$, \eqref{cl4} there remains true. On the other hand, for any given small $\epsilon>0$, we have
 \bes
 \sum_{l=N_0+1}^\infty|c_l-d_l|<\epsilon\ \ \ \mbox{for some large}\ N_0=N_0(\epsilon).
 \label{claim6}
 \ees
Furthermore, the open set $(a_1,z_i)\setminus\bigcup_{l=1}^{N_0}[c_l,d_l])$ consists of $N^*=N_0+1$ open intervals, denoted by
$\bigcup_{l=1}^{N^*}(e_l,f_l))$, with each $(e_l,f_l)\subset(a_1,z_i)$ and $[e_l,f_l]\cap[e_j,f_j]=\emptyset,\,\forall l\not=j$. Clearly,
$\mathcal{A}\cap(f_l-\rho,f_l)\not=\emptyset,\,\forall 1\leq l\leq N^*$ for any small $\rho>0$.

For simplicity, denote
 $$
 \mathcal{B}(\epsilon)=(a_1,z_i)\setminus\Big(\mathcal{A}\cup\bigcup_{l=1}^{N_0(\epsilon)}[c_l,d_l]\Big).
 $$
For the above $\epsilon$, we then claim that, given $1\leq l\leq N^*$, there is a large $k^l=k^l(\epsilon)$, such that
 \bes
 \label{claim7}
 d(x,\mathcal{A}_k\cap(x,f_l))\leq\Big(1+{1\over{N^*}}\Big)\epsilon,\ \ \forall x\in\mathcal{B}\cap(e_l,f_l-\epsilon/{N^*}),\ \ \mbox{for all}\ k\geq k^l.
 \ees
On the contrary, suppose that \eqref{claim7} is invalid. Then for some $1\leq l_0\leq N^*$, there is a point sequence
$x_{_{k_\eta}}\in\mathcal{B}\cap(e_{l_0},f_{l_0}-\epsilon/{N^*})$, such that
 \bes
 \nonumber
 d(x_{_{k_\eta}},\mathcal{A}_{_{k_\eta}}\cap(x_{_{k_\eta}},f_{l_0}))>\Big(1+{1\over{N^*}}\Big)\epsilon,\ \ \ \mbox{for all}\ \eta\geq 1.
 \ees
We may assume that $x_{_{k_\eta}}\to x^*\in\overline{\mathcal{B}}\cap[e_{l_0},f_{l_0}-\epsilon/{N^*}]$.
As in the proof of Lemma \ref{l2.4}, we have by sending $\eta\to\infty$ that
 $$
 d(x^*,\mathcal{A}\cap(x^*,f_{l_0}))\geq\Big(1+{1\over{N^*}}\Big)\epsilon.
 $$

If $d(x^*,\mathcal{A}\cap(x^*,f_{l_0}))=\infty$, obviously $m$ is constant on $[x^*,f_{l_0}]$, contradicting against the fact
that $\mathcal{A}\cap(f_l-\rho,f_l)\not=\emptyset,\,\forall 1\leq l\leq N^*$ for any small $\rho>0$.
If $(1+{1\over{N^*}})\epsilon\leq d(x^*,\mathcal{A}\cap(x^*,f_{l_0}))<\infty$, then it is necessary that
$m$ is constant on $(x^*,x^*+(1+{1\over{N^*}})\epsilon)\subset[e_{l_0},f_{l_0}]$.
This is also impossible because \eqref{claim6} and the definition of $[e_{l_0},f_{l_0}]$
already imply that $f_{l_0}-e_{l_0}<\epsilon$.

Next we are going to show
 \bes
 \label{claim8}
  \limsup_{\epsilon\to0}\limsup_{j\to\infty}\Big(\epsilon^{-{1\over3}}\|w(s_j,\cdot)\|_{L^\infty(\mathcal{B}\cap(\bigcup_{l=1}^{N^*}(e_l,f_l-\epsilon/{N^*}))}\Big)< 2.
 \ees
Otherwise, by means of \eqref{claim7} and \eqref{cl4}, we can find a sequence $\epsilon_{_{j_\kappa}}$ satisfying $\epsilon_{_{j_\kappa}}\to0$ as $\kappa\to\infty$ and two points sequences $x_{_{j_\kappa}}\in\mathcal{B}(\epsilon_{_{j_\kappa}})\cap(\bigcup_{l=1}^{N^*(\epsilon_{_{j_\kappa}})}(e_l(\epsilon_{_{j_\kappa}}),f_l(\epsilon_{_{j_\kappa}})-\epsilon_{_{j_\kappa}}/{N^*(\epsilon_{_{j_\kappa}})})$ and
$y_{_{j_\kappa}}\in\mathcal{A}$ satisfying $0<y_{_{j_\kappa}}-x_{_{j_\kappa}}<(1+{1\over{N^*(\epsilon_{_{j_\kappa}})}})\epsilon_{_{j_\kappa}}\leq2\epsilon_{_{j_\kappa}}$ such that
$w(s_{_{j_\kappa}},x_{_{j_\kappa}})>(\epsilon_{_{j_\kappa}})^{1/3}$ and $w(s_{_{j_\kappa}},y_{_{j_\kappa}})<\epsilon_{_{j_\kappa}},\,\forall \kappa\geq1$.

Given large $\kappa$, by the argument similar to that of deriving the claim \eqref{claim1}, we can find a function $W$ such that
 \begin{itemize}
 \item $W$ is strictly decreasing on its domain $[a_\infty,b_\infty]$ with
 $0<b_\infty-a_\infty\leq y_{_{j_\kappa}}-x_{_{j_\kappa}}\leq2\epsilon_{_{j_\kappa}}$, is continuous at $a_\infty$ and $b_\infty$,
 and is $C^1([a_\infty,b_\infty]\setminus\mathcal{O})$ with $\mathcal{O}$ containing at most countably many points;

 \item $W(b_\infty)-W(a_\infty)\geq w(s_{_{j_\kappa}},x_{_{j_\kappa}})-w(s_{_{j_\kappa}},y_{_{j_\kappa}})>(\epsilon_{_{j_\kappa}})^{1/3}-\epsilon_{_{j_\kappa}}>0$;

 \item $\int_{a_\infty}^{b_\infty}|W'(x)|^2dx\leq\sum_{i\in\mathbb{N}}\int_{\hat a_i}^{\hat b_i}|w'(s_j,x)|^2dx$, where
$\bigcup_{i\in\mathbb{N}}(\hat a_i,\hat b_i)=\{x\in(x_{_{j_\kappa}},y_{_{j_\kappa}}): \ w'(s_{_{j_\kappa}},x)<0\}$, with $\mathbb{N}$ being a given
set consisting of at most countably many integers, $(\hat a_i,\hat b_i)\cap(\hat a_j,\hat b_j)=\emptyset,\,\forall i\not=j$.

 \end{itemize}
Using $m'(x)\geq0$ on $[x_{_{j_\kappa}},y_{_{j_\kappa}}]$ and $w'(s_{_{j_\kappa}},x)\leq0$ on each $[\hat a_i,\hat b_i]$, the same analysis
as in step 3 of the proof of Lemma \ref{l2.4} deduces
 \bes
 2\max_{[0,1]}|c(x)|\geq {{|w(s_{_{j_\kappa}},y_{_{j_\kappa}})-w(s_{_{j_\kappa}},x_{_{j_\kappa}})|^2}\over{y_{_{j_\kappa}}-x_{_{j_\kappa}}}}
 \geq{{[(\epsilon_{_{j_\kappa}})^{1/3}-\epsilon_{_{j_\kappa}}]^2}\over{2\epsilon_{_{j_\kappa}}}}\to\infty,
 \nonumber
 \ees
as $\kappa\to\infty$. This contradiction yields \eqref{claim8}.

We recall that Lemma \ref{l2.3} implies $\mu(\mathcal{A}_k)=0,\forall k\geq1$
and therefore $\mu(\mathcal{A})=0$. Combing this fact, \eqref{claim5}, \eqref{claim6} and \eqref{claim8},
from \eqref{1.7} it is not hard to see that
 \bes
 \left.\begin{array}{lll}
 \medskip
 \displaystyle
 \mu((a_1,z_i)\setminus\bigcup_{l=1}^{L}[c_l,d_l])&\leq&
 \medskip
 \displaystyle
 \mu((a_1,z_i)\setminus\bigcup_{l=1}^{N_0}[c_l,d_l])=\mu(\mathcal{B}\cap\bigcup_{l=1}^{N^*}(e_l,f_l))\\
 \medskip
 &=&
 \displaystyle
 \mu(\mathcal{B}\cap\bigcup_{l=1}^{N^*}(e_l,f_l-\epsilon/{N^*}))+\mu(\mathcal{B}\cap\bigcup_{l=1}^{N^*}[f_l-\epsilon/{N^*},f_l))\\
 \medskip
 &\leq&
 \displaystyle
 \sum_{l=1}^{N^*}\mu(\mathcal{B}\cap(e_l,f_l-\epsilon/{N^*}))+\sum_{l=1}^{N^*}\mu([f_l-\epsilon/{N^*},f_l))\\
 \medskip
 &\leq&
 \displaystyle
 4\epsilon^{2/3}+2M^2\epsilon.
 \end{array}
 \right.
 \nonumber
 \ees
As $\epsilon$ can be arbitrarily small, we get \eqref{l2.4a-2}, and so \eqref{l2.4a-1} holds.

In what follows, we will show $\mu(\{a_1\})=0$. Assume that $a_1\in(0,1)$.
Since $m$ is non-decreasing on $[a_1,a_2]$ and $(a_1,a_1+\epsilon)\cap\{x\in[0,1]:\ m'(x)>0\}\not=\emptyset,\,\forall \epsilon>0$, we know from \eqref{a} that
there are two cases to occur: Case (i)\ \ $m'(x)\geq0$ on $[a_1-\delta_0,a_1]$ for some small $\delta_0$;
Case (ii)\ \ $m'(x)\leq0$ on $[a_1-\delta_0,a_1]$ for some small $\delta_0$ and
$\{m'(x)<0\}\cap(a_1-\epsilon,a_1)\not=\emptyset$ for any small $\epsilon$.

In each case, by the same argument as obtaining the claim \eqref{claim1} in the proof of
Lemma \ref{l2.4} we have
 \bes
 \label{claim9}
 \limsup_{j\to\infty}\|w(s_j,\cdot)\|_{L^\infty((a_1-\delta_0,a_1+\delta_0))}=m_*<\infty.
 \ees
As a consequence, for any small $0<\epsilon<\delta_0$, by taking $\zeta=1$ on $[a_1-\epsilon,a_1+\epsilon]$,
$\zeta=0$ on $[0,a_1-2\epsilon]\cup[a_1+2\epsilon,1]$, and $0\leq\zeta\leq1$ otherwise so that $\zeta\in C([0,1])$,
from \eqref{1.7} it easily follows that
 \bes
 \left.\begin{array}{l}
 \displaystyle
 \mu(\{a_1\})\leq\mu([a_1-\epsilon,a_1+\epsilon])\leq\lim\limits_{j\to\infty}\int_0^1w^2(s_j,x)\zeta(x)dx\leq 4m_*^2\epsilon,
 \end{array}
 \right.
 \nonumber
 \ees
which implies $\mu(\{a_1\})=0$.

Clearly the above analysis can be used to handle the case $a_1=0$ and the end points $c_l,\,d_l$
of the interval $[c_l,d_l]$ for each $l\geq1$, and obtain $\mu(\{c_l\})=\mu(\{d_l\})=0$. There are at most
countably many such end points. Thus, the countable additivity of the Radon measure ensures $\mu([a_1,a_2)\setminus\bigcup_{l=1}^{L}(c_l,d_l))=0$.
If additionally $0<a_2<1$ and $m'(x)\geq0$ on $[a_2,a_2+\epsilon_0]$ for some small $\epsilon_0>0$, by the analysis as above we have
$\mu(\{a_2\})=0$. The proof of Lemma \ref{l2.4a} is ended.

\end{proof}

\begin{re}\label{r-a} The parallel assertion to Lemma \ref{l2.4a} holds: Assume that $m$ is non-increasing on $[a_1,a_2]\subset[0,1]$,
$(a_1,a_1+\epsilon)\cap\{x\in[0,1]:\ m'(x)<0\}\not=\emptyset$ and $(a_2-\epsilon,a_2)\cap\{x\in[0,1]:\ m'(x)<0\}\not=\emptyset,\,\forall \epsilon>0$.
Denote by $\bigcup_{l=1}^{L}(c_l,d_l)$ the set of all the intervals that $m$ is constant on each $[c_l,d_l]\subset(a_1,a_2)$,
and $[c_l,d_l]\cap[c_j,d_j]=\emptyset,\,\forall l\not=j$,
where ${L}\geq0$ is a finite integer or equal to $\infty$. Then we have $\mu((a_1,a_2]\setminus\bigcup_{l=1}^{L}(c_l,d_l))=0$,
and if additionally $0<a_1<1$ and $m'(x)\leq0$ on $[a_1-\epsilon_0,a_1]$ for some small $\epsilon_0>0$, then
$\mu([a_1,a_2]\setminus\bigcup_{l=1}^{L}(c_l,d_l))=0$.

\end{re}

With the above preparation, we are now in a position to prove Theorem \ref{th1.2} in the case of Neumann boundary condition.
In order to avoid using complicated notation as well as tedious analysis, we consider two special functions
$m$ which carry typical kinds of degeneracy, and derive the limit $\lim\limits_{s\to\infty}\lambda_1^\mathcal{N}(s)$
for each such chosen $m$. Then, it will be not hard to see that Theorem \ref{th1.2} follows by
using the similar argument as for those two cases. Keeping such a strategy in mind, we first choose a function $m$ such that

(T1): \ $m$ is strictly increasing on $[0,a_1]\cup[a_2,a_3]\cup[a_5,1]$, and is strictly decreasing on $[a_1,a_2]\cup[a_4,a_5]$, and
$m$ is constant on $[a_3,a_4]$, for positive constants $0<a_1<a_2<a_3<a_4<a_5<1$. See Figure 7.
For such given $m$, we can conclude that

\begin{figure}[htbp]\label{figure4}
\centering {\includegraphics[height=1.40in]{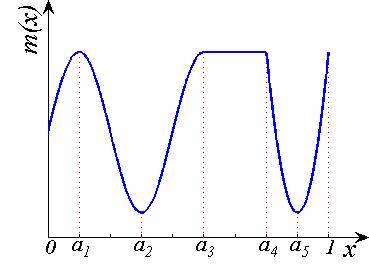}}\ \ \ \ \ \ \  \ \ \ \ \ \ \ \ \ \ \ \  \ \ \
{\includegraphics[height=1.40in]{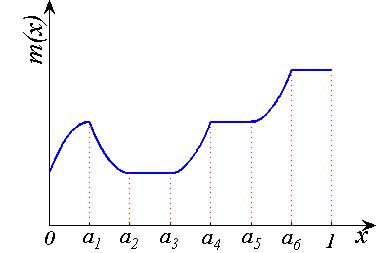}}

{Figure 7 \ \ \ \ \ \ \ \ \ \ \ \ \ \ \ \ \ \ \ \ \ \ \ \ \  \ \  \ \  \ \ \  \ \ \ \ \ \ \ \ \ \ \ \ \ \ \ \ \ \ \ Figure 8}
\end{figure}

\begin{theo}\label{th2.1} Assume that \rm{(T1)} holds, then
 $$
 \lim_{s\to\infty}\lambda_1^\mathcal{N}(s)=\min\{c(a_1),\ c(1),\ \lambda_1^\mathcal{NN}(a_3,a_4)\}.
 $$

\end{theo}

\begin{proof} By Lemma \ref{l2.1}, we have
 \bes
 \limsup_{s\to\infty}\lambda_1^\mathcal{N}(s)=\lambda^*\leq\min\{c(a_1),\ c(1),\ \lambda_1^\mathcal{NN}(a_3,a_4)\}.
 \label{2.8c}
 \ees
So it suffices to prove
 \bes
 \liminf_{s\to\infty}\lambda_1^\mathcal{N}(s)=\lambda_*\geq\min\{c(a_1),\ c(1),\ \lambda_1^\mathcal{NN}(a_3,a_4)\}.
 \label{2.9}
 \ees

First of all, Lemmas \ref{l2.4}, \ref{l2.4a}, Remark \ref{r-a} and the countable additivity of the probability measure tell us that
 \bes
 \mu([0,a_1)\cup(a_1,a_3]\cup[a_4,1))=0.
 \label{2.9a}
 \ees
Thus, there holds
 \bes
 \mu(\{a_1\})+\mu((a_3,a_4))+\mu(\{1\})=1.
 \label{2.10}
 \ees

On the other hand, for any small $\epsilon>0$, from \eqref{1.4} and the assumption that $m$ is
constant on $[a_3,a_4]$ it follows that
 \bes
 \left.\begin{array}{lll}
 \medskip
 \displaystyle
 \lambda_1^\mathcal{N}(s_j)
 &=&\displaystyle
 \int_0^1\{|w'(s_j,x)-s_jw(s_j,x)m'(x)|^2+c(x)w^2(s_j,x)\}dx\\
 \medskip
 &\geq&
 \displaystyle
 \Big(\int_0^{a_1-\epsilon}+\int_{a_1+\epsilon}^{a_3}+\int_{a_4}^{1-\epsilon}\Big)c(x)w^2(s_j,x)dx
 +\Big(\int_{a_1-\epsilon}^{a_1+\epsilon}+\int_{1-\epsilon}^{1}\Big)c(x)w^2(s_j,x)dx\\
 \medskip
 &&\ +
 \displaystyle
 \int_{a_3}^{a_4}\{|w'(s_j,x)|^2+c(x)w^2(s_j,x)\}dx.
 \end{array}
 \right.
 \label{2.11}
 \ees
Appealing to Lemma \ref{l2.5}(ii), we get
 \bes
 \Big(\int_0^{a_1-\epsilon}+\int_{a_1+\epsilon}^{a_3}+\int_{a_4}^{1-\epsilon}\Big)c(x)w^2(s_j,x)dx\to0,\ \ \mbox{as}\ j\to\infty,
 \label{2.11a}
 \ees
and by Lemma \ref{l2.6}, we have
 \bes
 \label{2.11b}\ \  \ \ \ \ \
 \lim_{j\to\infty}\Big(\int_{a_1-\epsilon}^{a_1+\epsilon}+\int_{1-\epsilon}^{1}\Big)c(x)w^2(s_j,x)dx
 \geq\min_{[a_1-\epsilon/2,a_1+\epsilon/2]}c(x)\, \mu(\{a_1\})+\min_{[1-\epsilon/2,1]}c(x)\, \mu(\{1\}).
 \ees

To handle the last integral in \eqref{2.11}, let us assume, for the moment, that there is a sequence of $\{w(s_j,\cdot)\}$,
still labelled by itself, such that
 \bes
 \label{2.11c}
 \int_{a_3}^{a_4}w^2(s_j,x)dx>0, \ \ \mbox{for all}\ j\geq 1,
 \ees
and hence set
 $$
 v(s_j,x)={{w(s_j,x)}\over{\sqrt{\int_{a_3}^{a_4}w^2(s_j,x)dx}}}.
 $$
Thus, $v(s_j,x)$\ satisfies $\int_{a_3}^{a_4}z^2(s_j,x)dx=1$ for each $j\geq1$.
By the variational characterization of $\lambda_1^\mathcal{NN}(a_3,a_4)$, combined with \eqref{2.9a} and Lemma \ref{l2.5}(iii), one finds
 \bes
 \left.\begin{array}{lll}
 \medskip
 \displaystyle
 \int_{a_3}^{a_4}\{|w'(s_j,x)|^2+c(x)w^2(s_j,x)\}dx
 &=&\displaystyle
 \int_{a_3}^{a_4}w^2(s_j,x)dx\int_{a_3}^{a_4}\{|v'(s_j,x)|^2+c(x)v^2(s_j,x)\}dx\\
 \medskip
 &\geq&
 \displaystyle
 \lambda_1^\mathcal{NN}(a_3,a_4)\int_{a_3}^{a_4}w^2(s_j,x)dx\\
 \medskip
 &\to&
 \displaystyle
 \lambda_1^\mathcal{NN}(a_3,a_4)\mu((a_3,a_4)),\ \ \mbox{as}\ j\to\infty.
 \end{array}
 \right.
 \label{2.12}
 \ees
By means of \eqref{2.11}, \eqref{2.11a}, \eqref{2.11b} and \eqref{2.12}, we deduce
 $$
 \lambda_*=\liminf_{j\to\infty}\lambda_1^\mathcal{N}(s_j)\geq\min_{[a_1-\epsilon/2,a_1+\epsilon/2]}c(x)\,
 \mu(\{a_1\})+\min_{[1-\epsilon/2,1]}c(x)\, \mu(\{1\})+\lambda_1^\mathcal{NN}(a_3,a_4)\mu((a_3,a_4)).
 $$
Since $\epsilon>0$ is arbitrary, by sending $\epsilon\to0$ in the above inequality and using \eqref{2.10}, we obtain
 \bes
 \left.\begin{array}{lll}
 \medskip
 \displaystyle
 \lambda_*
 &\geq&\displaystyle
 c(a_1)\mu(\{a_1\})+c(1)\mu(\{1\})+\lambda_1^\mathcal{NN}(a_3,a_4)\mu((a_3,a_4))\\
 \medskip
 &\geq&
 \displaystyle
 \min\{c(a_1),\ c(1),\ \lambda_1^\mathcal{NN}(a_3,a_4)\}.
 \end{array}
 \right.
 \label{2.13}
 \ees

If \eqref{2.11c} does not hold, then $\int_{a_3}^{a_4}w^2(s_j,x)dx=0$ for all large $j$,
and in turn $w(s_j,\cdot)\equiv0$ on $[a_3,a_4]$ since $w(s_j,\cdot)\in C([a_3,a_4])$. Hence, the last integral in \eqref{2.11} converges to zero.
In addition, Lemma \ref{l2.5}(i) asserts $\mu([a_3,a_4])=0$, and so $\mu(\{a_1\})+\mu(\{1\})=1$ due to \eqref{2.10}.
Using these facts, \eqref{2.11a} and \eqref{2.11b}, it follows from \eqref{2.11} and \eqref{2.8c} that
 \bes
 \min\{c(a_1),\ c(1),\ \lambda_1^\mathcal{NN}(a_3,a_4)\}\geq\lambda^*\geq\lambda_*\geq\displaystyle c(a_1)\mu(\{a_1\})+c(1)\mu(\{1\})\geq \min\{c(a_1),\ c(1)\},
 \nonumber
 \ees
which thereby implies $\lambda_1^\mathcal{NN}(a_3,a_4)\geq\min\{c(a_1),\ c(1)\}$. As a result, \eqref{2.13} remains true and \eqref{2.9} is established.
This completes the proof.

\end{proof}

\begin{re}\label{r1} A direct check of the proof of Theorem \ref{th2.1} shows that the support of the probability measure $\mu$ is the set
where the limit $\lim\limits_{s\to\infty}\lambda_1^\mathcal{N}(s)$ is attained; for example, if $\lim\limits_{s\to\infty}\lambda_1^\mathcal{N}(s)=c(a_1)<\min\{c(1),\ \lambda_1^\mathcal{NN}(a_3,a_4)\}$, then $\mu(\{a_1\})=1$ and $\mu([0,a_1)\cup(a_1,1])=0$, and if $\lim\limits_{s\to\infty}\lambda_1^\mathcal{N}(s)=c(a_1)=c(1)<\lambda_1^\mathcal{NN}(a_3,a_4)$, then $\mu(\{a_1,\,1\})=1$ and $\mu([0,a_1)\cup(a_1,1))=0$. Conversely, the limit is attained on the support of $\mu$; for example, if $\mu(\{a_1\})>0$, then $\lim\limits_{s\to\infty}\lambda_1^\mathcal{N}(s)=c(a_1)$ (and so $c(a_1)\leq\min\{c(1),\ \lambda_1^\mathcal{NN}(a_3,a_4)\}$), and if $\mu((a_3,a_4))>0$, then $\lim\limits_{s\to\infty}\lambda_1^\mathcal{N}(s)=\lambda_1^\mathcal{NN}(a_3,a_4)$ (and so $\lambda_1^\mathcal{NN}(a_3,a_4)\leq\min\{c(a_1),\ c(1)\}$). A similar comment also applies to Theorem \ref{th2.2} below.

\end{re}

We next choose another typical function $m$ satisfying

(T2): \ $m$ is strictly increasing on $[0,a_1]\cup[a_3,a_4]\cup[a_5,a_6]$, and is strictly decreasing on $[a_1,a_2]$, and
$m$ is constant on $[a_2,a_3]$, $[a_4,a_5]$ and $[a_6,1]$, for positive constants $0<a_1<a_2<a_3<a_4<a_5<a_6<1$. See Figure 8.
For such a given function $m$, we have

\begin{theo}\label{th2.2} Assume that \rm{(T2)} holds, then
 $$
 \lim_{s\to\infty}\lambda_1^\mathcal{N}(s)=\min\{c(a_1),\ \lambda_1^\mathcal{DD}(a_2,a_3),\ \lambda_1^\mathcal{ND}(a_4,a_5),\ \lambda_1^\mathcal{NN}(a_6,1)\}.
 $$

\end{theo}

\begin{proof} By Lemma \ref{l2.1}, we have
 \bes
 \lambda^*\leq\min\{c(a_1),\ \lambda_1^\mathcal{DD}(a_2,a_3),\ \lambda_1^\mathcal{ND}(a_4,a_5),\ \lambda_1^\mathcal{NN}(a_6,1)\}.
 \label{2.14}
 \ees
In addition, applying Lemmas \ref{l2.4}, \ref{l2.4a} and Remark \ref{r-a}, we obtain
 \bes
 \mu([0,a_1)\cup(a_1,a_2]\cup[a_3,a_4]\cup[a_5,a_6])=0
 \label{2.16}
 \ees
and so
 \bes
 \mu(\{a_1\})+\mu((a_2,a_3))+\mu((a_4,a_5))+\mu(a_6,1])=1.
 \label{2.17}
 \ees
In what follows we are going to prove
 \bes
 \lambda_*\geq\min\{c(a_1),\ \lambda_1^\mathcal{DD}(a_2,a_3),\ \lambda_1^\mathcal{ND}(a_4,a_5),\ \lambda_1^\mathcal{NN}(a_6,1)\}.
 \label{2.18}
 \ees

\vskip6pt To achieve the aim, we first need the following fact: for a generic sequence $\{s_j\}$, there holds
 \bes
 w(s_j,x)\to w_*(x)\ \ \mbox{uniformly on\ $[a_2,a_3]$},\ \ \mbox{as}\ j\to\infty,
 \label{2.8}
 \ees
for some $w_*\in C([a_2,a_3])$. Indeed, since $m(x)$ is constant on $[a_2,a_3]$, we get
 \bes
 \left.\begin{array}{lll}
 \medskip
 \displaystyle
 \max_{[0,1]}|c(x)|\geq\lambda_1^\mathcal{N}(s_j)&=&
 \medskip
 \displaystyle
 \int_0^1[|w'(s_j,x)-s_jw(s_j,x)m'(x)|^2+c(x)w^2(s_j,x)]dx\\
 &\geq&
 \displaystyle
 \int_{a_2}^{a_3}|w'(s_j,x)|^2dx-\max_{[0,1]}|c(x)|,
 \end{array}
 \right.
 \nonumber
 \ees
which in turn gives
 $$
 \int_{a_2}^{a_3}w^2(s_j,x)dx+\int_{a_2}^{a_3}|w'(s_j,x)|^2dx\leq1+2\max_{[0,1]}|c(x)|.
 $$
Notice that $H^1((a_2,a_3))$ is compactly embedded into $C([a_2,a_3])$. Thus, \eqref{2.8} holds true.

As $m$ is constant on $[a_2,a_3]$ and $w(s_j,\cdot)$ satisfies
 \bes
 -w''(s_j,x)+c(x)w(s_j,x)=\lambda_1^\mathcal{N}(s_j)w(s_j,x),\ a_2<x<a_3.
 \label{xx}
 \ees
By a standard compactness consideration, for a subsequence of $\{s_j\}$, denoted by itself for simplicity,
satisfying $\lambda_1^\mathcal{N}(s_j)\to\lambda_*$ as $j\to\infty$, it is easily seen that
$w(s_j,x)\to w_*(x)$ in $C^1_{{loc}}((a_2,a_3))$. Combining this fact and \eqref{2.8}, one can use \eqref{xx} to further
conclude that $w(s_j,x)\to w_*(x)$ in $C^1([a_2,a_3])$. So $w_*\in C^1([a_2,a_3])$ solves in the weak sense
 \bes
 -w_*''(x)+c(x)w_*(x)=\lambda_*w_*(x),\ \ a_2<x<a_3.
 \nonumber
 \ees

In the sequel, we will show $w_*(a_2)=w_*(a_3)=0$. We argue indirectly and suppose that $w_*(a_2)>0$.
Since $w(s_j,a_2)\to w_*(a_2)>0$, $w(s_j,a_2)\geq {1\over2}w_*(a_2)$ for all large $j$. On the other hand,
thanks to $\{m'(x)<0\}\cap(a_2-\rho,a_2)\not=\emptyset,\,\forall \rho>0$, for any given small constant $\delta$ with
$0<\delta<{1\over4}w_*(a_2)$, there exist two sequences $\{j_k\}_{k=1}^\infty$ and $\{\rho_k\}_{k=1}^\infty$
satisfying $j_k\to\infty,\,\rho_k\to0$ as $k\to\infty$, such that
$w(s_{_{j_k}},a_2-\rho_k)\leq\delta,\,\forall k\geq1$. Noticing that $\rho_k\to0$ as $k\to\infty$, and
 $$
 w(s_{_{j_k}},a_2)\geq {1\over2}w_*(a_2),\ \ w(s_{_{j_k}},a_2-\rho_k)\leq\delta<{1\over4}w_*(a_2),
 $$
one can use the similar analysis to that in the proof of Lemma \ref{l2.4} to arrive at a contradiction. Thus $w_*(a_2)=0$ holds. Similarly, we have
$w_*(a_3)=0$.

To summarize, the above argument asserts that $w(s_j,x)\to w_*(x)$ in $C^1([a_2,a_3])$
for some function $w_*$, where $w_*\geq0$ solves
 \bes
 -w_*''(x)+c(x)w_*(x)=\lambda_*w_*(x),\ a_2<x<a_3;\ \ \ w_*(a_2)=w_*(a_3)=0.
 \label{2.21}
 \ees
Similarly, on $[a_4,a_5]$ and $[a_6,1]$, by passing to a further subsequence if necessary, we know that
$w(s_j,x)\to w_*(x)$ in $C^1([a_4,a_5])$, and $w_*\geq0$ in $[a_4,a_5]$ and solves
 \bes
 -w_*''(x)+c(x)w_*(x)=\lambda_*w_*(x),\ a_4<x<a_5;\ \ w_*(a_5)=0,
 \label{2.22}
 \ees
and $w(s_j,x)\to w_*(x)$ in $C^1([a_6,1])$, and $w_*\geq0$ in $[a_6,1]$ and solves
 \bes
 -w_*''(x)+c(x)w_*(x)=\lambda_*w_*(x),\ a_6<x<1.
 \label{2.23}
 \ees

\vskip6pt
Now, given any small $\epsilon>0$, by means of \eqref{1.4} and the assumption on $m$, we obtain
 \bes
 \left.\begin{array}{lll}
 \medskip
 \displaystyle
 \lambda_1^\mathcal{N}(s_j)
 &=&\displaystyle
 \int_0^1\{|w'(s_j,x)-s_jw(s_j,x)m'(x)|^2+c(x)w^2(s_j,x)\}dx\\
 \medskip
 &\geq&
 \displaystyle
 \Big(\int_0^{a_1-\epsilon}+\int_{a_1+\epsilon}^{a_2}+\int_{a_3}^{a_4}+\int_{a_5}^{a_6}\Big)c(x)w^2(s_j,x)dx
 +\int_{a_1-\epsilon}^{a_1+\epsilon}c(x)w^2(s_j,x)dx\\
 \medskip
 &&\ +
 \displaystyle
 \Big(\int_{a_2}^{a_3}+\int_{a_4}^{a_5}+\int_{a_6}^{1}\Big)\Big\{|w'(s_j,x)|^2+c(x)w^2(s_j,x)\Big\}dx.
 \end{array}
 \right.
 \label{2.24}
 \ees
Using Lemma \ref{l2.5}(ii), it is easily seen that the first term in \eqref{2.24} converges to zero as $j\to\infty$.
By virtue of Lemma \ref{l2.6}, one has
 \bes
 \label{claim14}
 \lim_{j\to\infty}\int_{a_1-\epsilon}^{a_1+\epsilon}c(x)w^2(s_j,x)dx
 \geq\min_{[a_1-\epsilon/2,a_1+\epsilon/2]}c(x)\, \mu(\{a_1\}).
 \ees

Assume that there is a subsequence of $\{s_j\}$, labelled by itself again, such that
  \bes
 \label{2.24b}
 \int_{a_6}^{1}w^2(s_j,x)dx>0 \ \ \mbox{for all large}\ j,
 \ees
and
 \bes
 \label{2.24a}
 w_*(x)\geq,\not\equiv0 \ \ \mbox{on}\ [a_2,a_3],\ \  w_*(x)\geq,\not\equiv0 \ \ \mbox{on}\ [a_4,a_5].
 \ees
Let us denote
 $$
 v_1(s_j,x)={{w(s_j,x)}\over{\sqrt{\int_{a_2}^{a_3}w^2(s_j,x)dx}}},\ \ v_2(s_j,x)={{w(s_j,x)}\over{\sqrt{\int_{a_4}^{a_5}w^2(s_j,x)dx}}},
 \  \ v_3(s_j,x)={{w(s_j,x)}\over{\sqrt{\int_{a_6}^{1}w^2(s_j,x)dx}}}.
 $$
Thus, we have
 $$
 v_1\to{{w_*}\over{\sqrt{\int_{a_2}^{a_3}w^2_*(x)dx}}}=:w_1\ \mbox{in}\ C^1([a_2,a_3]),\ \ v_2\to{{w_*}\over{\sqrt{\int_{a_4}^{a_5}w^2_*(x)dx}}}=:w_2\ \mbox{in}\ C^1([a_4,a_5]),\ \
 $$
and $\int_{a_6}^{1}v_3^2(s_j,x)dx=1$ for each $j\geq1$, where $w_1,\,w_2\geq,\not\equiv0$ with $w_1(a_2)=w_1(a_3)=w_2(a_5)=0$ and
$\int_{a_2}^{a_3}w_1^2(x)dx=\int_{a_4}^{a_5}w_2^2(x)dx=1$.

In view of the variational characterizations for $\lambda_1^\mathcal{DD}(a_2,a_3),\ \lambda_1^\mathcal{ND}(a_4,a_5),\ \lambda_1^\mathcal{NN}(a_6,1)$,
similarly to the proof of Theorem \ref{th2.1} we have, as $j\to\infty$,
 \bes
 \left.\begin{array}{lll}
 \medskip
 \displaystyle
 \int_{a_2}^{a_3}\{|w'(s_j,x)|^2+c(x)w^2(s_j,x)\}dx
 &=&\displaystyle
 \int_{a_2}^{a_3}w^2(s_j,x)dx\int_{a_2}^{a_3}\{|v_1'(s_j,x)|^2+c(x)v_1^2(s_j,x)\}dx\\
 \medskip
 &\to&
 \displaystyle
 \int_{a_2}^{a_3}w^2_*(x)dx\int_{a_2}^{a_3}\{|w_1'(x)|^2+c(x)w_1^2(x)\}dx\\
 \medskip
 &\geq&
 \displaystyle
 \lambda_1^\mathcal{DD}(a_2,a_3)\mu((a_3,a_4)),
 \end{array}
 \right.
 \label{2.25a}
 \ees
and
 \bes
 \int_{a_4}^{a_5}\{|w'(s_j,x)|^2+c(x)w^2(s_j,x)\}dx\geq\lambda_1^\mathcal{ND}(a_4,a_5)\mu((a_4,a_5)),
 \label{2.25b}
 \ees
and
 \bes
 \int_{a_6}^{1}\{|w'(s_j,x)|^2+c(x)w^2(s_j,x)\}dx\geq\lambda_1^\mathcal{NN}(a_6,1)\mu((a_6,1]).
 \label{2.25c}
 \ees
Making use of \eqref{2.17}, \eqref{2.24}, \eqref{claim14}, \eqref{2.25a}, \eqref{2.25b} and \eqref{2.25c}, we can easily deduce \eqref{2.18}
when \eqref{2.24a} and \eqref{2.24b} hold.

If one of \eqref{2.24a} and \eqref{2.24b} is not satisfied, the analysis of
Theorem \ref{th2.1} can be adapted to obtain \eqref{2.18}. So the desired limit is derived.
\end{proof}

Following the similar argument to that of Theorems \ref{th2.1} and \ref{th2.2},
taking into account Lemma \ref{l2.4a} and its proof, one can obtain Theorem \ref{th1.2} for problem \eqref{1.1}.

\subsection{The
eigenvalue problems \eqref{p} and \eqref{1.p1}}
In the previous subsection, we have proved Theorem \ref{th1.2} in the Neumann boundary condition case (that is, $\ell_1=\ell_2=0$).
This subsection concerns problem \eqref{p} in the case of $\ell_1+\ell_2>0$ and the periodic problem \eqref{1.p1}. We first give

\vskip6pt
{\it Proof of Theorem \ref{th1.2} for $\ell_1+\ell_2>0$}\ \ We first consider the case of $\hbar_i,\,\ell_i>0\, (i=1,2)$. Then through the same transformation
$w=e^{sm}\varphi$ as for the Neumann problem, \eqref{p} is equivalent to
 \bes
 \left\{\begin{array}{ll}
 \medskip
 \displaystyle
 -w''(x)+[s^2(m'(x))^2+sm''(x)+c(x)]w(x)=\lambda_1(s)w(x),\ 0<x<1,\\
 \displaystyle
 (s\hbar_1+\ell_1)w(0)-\hbar_1w'(0)=(-s\hbar_2+\ell_2)w(1)+\hbar_2w'(1)=0.
 \end{array}
 \right.
 \label{4.1d}
 \ees
For each $s\in\bR$, $\lambda_1(s)$ enjoys the variational characterization
 \bes
 \left.\begin{array}{lll}
 \medskip
 \displaystyle
 \lambda_1(s)
 &=&
 \displaystyle
 \min_{\int_0^1 e^{2sm}\varphi^2dx=1}\Big\{\int_0^1e^{2sm}[(\varphi')^2+c\varphi^2]dx
 +{{\ell_1}\over{\hbar_1}}e^{2sm(0)}\varphi^2(0)+{{\ell_2}\over{\hbar_2}}e^{2sm(1)}\varphi^2(1)\Big\}\\
 &=&
 \displaystyle
 \min_{\int_0^1 w^2dx=1}\Big\{\int_0^1[(w'-swm')^2+cw^2]dx+{{\ell_1}\over{\hbar_1}}w^2(0)+{{\ell_2}\over{\hbar_2}}w^2(1)\Big\}.
 \end{array}
 \right.
 \nonumber
 \ees
Additionally, $\min_{x\in[0,1]}c(x)\leq\lambda_1(s),\,\forall s\in\bR$.

As in the Neumann case, we use $w(s,\cdot)$ with $\int_0^1w^2(s,x)dx=1$ to
denote the principal eigenfunction of \eqref{4.1d}, and by the weak compactness of
$\{w^2(s,\cdot)\}_{s>0}$, for a generic sequence
$\{s_j\}_{j=1}^\infty$ satisfying $s_j\to\infty$ as $j\to\infty$, we assume that
 \bes
 \lim\limits_{j\to\infty}\int_0^1w^2(s_j,x)\zeta(x)dx=\int_{[0,1]}\zeta(x)\mu(dx),\ \ \forall \zeta\in C([0,1])
 \nonumber
 \ees
for a unique probability measure $\mu$.

We assume that $m$ has at least one interior point of local maximum. Clearly, under our hypothesis \eqref{a},
$m$ contains at least one isolated interior point or segment of local maximum. Hence,
one can easily modify the argument of Lemma \ref{l2.1} to show that
$\limsup\limits_{s\to\infty}\lambda_1(s)<\infty$.

Conversely, if $m$ has no interior point of local maximum, then under the assumption \eqref{a}, there are only three
possibilities: $0$ is the unique isolated point of local maximum, and so $m$ must be strictly decreasing on $[0,1]$; \
$1$ is the unique isolated point of local maximum, and so $m$ must be strictly increasing on $[0,1]$; \
only $0$ and $1$ are the isolated points of local maximum, and so $m$ must be strictly decreasing on $[0,x_0]$
while strictly decreasing on $[x_0,1]$ for some $x_0\in(0,1)$. We only consider the first case and the other two cases can be tackled similarly.
In order to show $\lim\limits_{s\to\infty}\lambda_1(s)=\infty$, we proceed indirectly and suppose that
$\liminf\limits_{s\to\infty}\lambda_1(s)=\lambda_*\in[\min_{x\in[0,1]}c(x),\infty)$. Thus,
$\lim\limits_{j\to\infty}\lambda_1(s_j)=\lambda_*$ for some sequence $\{s_j\}_{j\geq1}$ satisfying $s_j\to\infty$
as $j\to\infty$. Then Lemma \ref{l2.4a} gives $\mu((0,1])=0$, and in turn $\mu(\{0\})=1$. By virtue of $\int_0^1w^2(s_i,x)dx=1$ for each $j\geq1$,
with the help of Lemma \ref{l2.5}(ii), one can further assert that
 \bes
 \nonumber
 \limsup_{j\to\infty}\|w(s_j,\cdot)\|_{L^\infty((0,1))}=\infty.
 \ees
Hence, the analysis similar to that of Lemma \ref{l2.4} results in a contradiction, and so it is necessary that
$\lim\limits_{s\to\infty}\lambda_1(s)=\infty$.

In summary, the above analysis shows that when $\hbar_i,\,\ell_i>0\, (i=1,2)$, $\lim\limits_{s\to\infty}\lambda_1(s)=\infty$
if and only if $\mathcal{M}=\mathcal{M}_1\subset\{0,1\}$. Moreover, if $\lim\limits_{s\to\infty}\lambda_1(s)<\infty$,
we first use the same argument as in Lemma \ref{l2.1} to deduce
$\limsup\limits_{s\to\infty}\lambda_1(s)\leq\hat\lambda$,
where $\hat\lambda$ is the limiting value given in Theorem \ref{th1.2}(ii). So as above, we can claim that $\mu(\{0,\,1\}\cap\mathcal{M}_1)=0$.
Combined with this fact, the analysis of Theorems \ref{th2.1} and \ref{th2.2} can be easily adapted to prove Theorem \ref{th1.2}(ii); the details are omitted
here to avoid unnecessary repetition. Here the only point we want to stress is that when one of $\mathcal{M}_i (i=6,7,8,9)$ is nonempty,
for instance, if $[0,a^I]\in\mathcal{M}_6$, since $m$ is constant on $[0,a^I]$, through the transformation $w=e^{sm}\varphi$, we notice that $w$ satisfies the same
boundary condition as $\varphi$ at the endpoint $0$: $-\hbar_1w'(0)+\ell_1w(0)=0$.

When $\hbar_1=\hbar_2=0$ (so $\ell_1,\,\ell_2>0$), \eqref{1.1} becomes Dirichlet eigenvalue problem \eqref{4.1}. Thus, we have
 \bes
 \left.\begin{array}{lll}
 \medskip
 \displaystyle
 \lambda_1(s)
 &=&
 \displaystyle
 \min_{\varphi\in H^1_0,\, \int_0^1 e^{2sm}\varphi^2dx=1}\int_0^1e^{2sm}[(\varphi')^2+c\varphi^2]dx\\
 &=&
 \displaystyle
 \min_{w\in H^1_0,\, \int_0^1 w^2dx=1}\int_0^1[(w'-swm')^2+cw^2]dx,
 \end{array}
 \right.
 \nonumber
 \ees
and when $\hbar_1=0,\,\hbar_2,\,\ell_2>0$, we have
 \bes
 \left.\begin{array}{lll}
 \medskip
 \displaystyle
 \lambda_1(s)
 &=&
 \displaystyle
 \min_{\varphi\in H^1_*,\, \int_0^1 e^{2sm}\varphi^2dx=1}\int_0^1e^{2sm}[(\varphi')^2+c\varphi^2]dx++{{\ell_2}\over{\hbar_2}}\varphi^2(1)\\
 &=&
 \displaystyle
 \min_{w\in H^1_*,\, \int_0^1 w^2dx=1}\int_0^1[(w'-swm')^2+cw^2]dx+{{\ell_2}\over{\hbar_2}}w^2(1),
 \end{array}
 \right.
 \nonumber
 \ees
where $H^1_*=\{g\in H^1(0,1):\, g(0)=0\}$.  By such variational characterizations,
we can use the argument similar to the above to obtain the desired results.
The remaining cases can be handled similarly, and the detailed are omitted. This completes the proof of Theorem \ref{th1.2}.
\ \ \ \ \  $\hfill \Box$

\vskip6pt At last, we give the

{\it Proof of Theorem \ref{th1.2a}}\ \ Recall that the principal eigenvalue $\lambda_1^\mathcal{P}(s)$
of \eqref{1.p1} can be variationally characterized as
 \bes
 \left.\begin{array}{lll}
 \medskip
 \displaystyle
 \lambda_1^\mathcal{P}(s)
 &=&
 \displaystyle
 \min_{\varphi\,is 1-periodic,\, \int_0^1 e^{2sm}\varphi^2dx=1}\int_0^1e^{2sm}[(\varphi')^2+c\varphi^2]dx\\
 &=&
 \displaystyle
 \min_{w\,is 1-periodic,\, \int_0^1 w^2dx=1}\int_0^1[(w'-swm')^2+cw^2]dx.
 \end{array}
 \right.
 \nonumber
 \ees

 As before, making use of the transformation $w=e^{sm}\varphi$, we obtain from \eqref{1.p1} that
 \bes
 \left\{\begin{array}{ll}
 \medskip
 \displaystyle
 -w''(x)+[s^2(m'(x))^2+sm''(x)+c(x)]w(x)=\lambda_1^\mathcal{P}(s)w(x),\ &x\in\bR,\\
 \displaystyle
 w(x)=w(x+1),\ &x\in\bR.
 \end{array}
 \right.
 \nonumber
 \ees

Consequently, Theorem \ref{th1.2a} follows from the same analysis as that of Theorem \ref{th1.2}.
\ \ \ \ \  $\hfill \Box$

\section{The principal eigenfunction:\ Proof of Theorem \ref{th1.3}}

This section is devoted to the study of the asymptotic behavior of the principal eigenfunction of \eqref{1.3} as $s\to\infty$.
It is easily seen from the analysis in the previous section that the assertions (2)-(9) of Theorem \ref{th1.3} hold.
Hence, it remains to prove the assertion (1) of Theorem \ref{th1.3}. To the end, we use the techniques introduced in \cite{CL1}
with necessary modifications.

We first consider the case $x_0\in(0,1)$.} By our assumption, there exist small positive constants $a$ and $R$, and a large positive constant $A$, such that
 \bes
 |m'(x)|^2\geq a|x-x_0|^{2(k^*-1)},\ \ \forall x\in B(x_0,4R)\subset(0,1),
 \label{3.1}
 \ees
and
 \bes
 |m''(x)|\leq A|x-x_0|^{k^*-2},\ \ \forall x\in B(x_0,4R)\subset(0,1).
 \label{3.1a}
 \ees

As in \cite{CL1}, through scaling, without loss of generality we may assume that
 \bes
 \|m'(x)\|_{L^\infty(0,1)}+\|m''(x)\|_{L^\infty(0,1)}+\max_{[0,1]}c(x)-\min_{[0,1]}c(x)\leq {1\over2},
 \label{3.1b}
 \ees
and define for convenience
 $$
 q(s,x)=s^2|m'(x)|^2+sm''(x)+c(x)-\lambda_1^\mathcal{N}(s).
 $$
Then, we have $w''(s,x)=q(s,x)w(s,x)$.

Lemma 5.1 and Lemma 5.2 of \cite{CL1} remain true in our current situation, that is, we have

\begin{lem}\label{l3.1} There exists a constant $M>0$ such that $\|w(s,\cdot)\|_{L^\infty(0,1)}\leq Ms^{1/2},\, \forall s\geq1$.
\end{lem}

\begin{lem}\label{l3.2} Let $k,\,r$ be positive constants and $W$ be a $C^2$ function satisfying
 $$
 \Delta W(x)=Q(x)W(x)>0,\ \ \ Q(x)\geq2k^2,\ \ \forall x\in(-r,r).
 $$
Then, $W(0)\leq e^{1-kr}\max\{W(-r),\ W(r)\}$.
\end{lem}

By slightly modifying the argument of Lemma 5.3 of \cite{CL1}, one can use Lemma \ref{l3.2} to deduce

\begin{lem}\label{l3.3} Let $a$ and $R$ be as in \eqref{3.1}. Then for any $s\geq1$, there holds
 \bes
 w(s,x)\leq e^{1-s^{1\over2}(|x-x_0|-({3\over{as}})^{1\over{2k^*-2}})}\max_{B(x_0,2|x-x_0|)}w(s,y),\ \ \forall x\in B(x_0,2R).
 \label{3.2}
 \ees
\end{lem}

\vskip10pt With the help of the above lemmas, we are now ready to give

{\it Proof of Theorem \ref{th1.3}\,(1-i)}\ \ Clearly, under our assumption on $m$ at $x=x_0$, one can find positive constants $a$
and $R$ such that \eqref{3.1} holds. Thus, $w$ satisfies \eqref{3.2}.

For any fixed $s\geq1$, we denote
 $$
 W(x_0,s;y)=s^{-{1\over{2k^*}}}w(s,x_0+s^{-{1\over{k^*}}}y),
 $$
and
 $$
 Q(x_0,s;y)=s^{-{2\over{k^*}}}\Big\{s^2\Big|m'(x_0+s^{-{1\over{k^*}}}y)\Big|^2+sm''(x_0+s^{-{1\over{k^*}}}y)+c(x_0+s^{-{1\over{k^*}}}y)-\lambda_1^\mathcal{N}(s)\Big\}.
 $$
Then, due to \eqref{3.1}, \eqref{3.1a} and \eqref{3.1b}, simple computation gives
 $$
 \Delta_yW=Q(x_0,s;y)W,\ \  \ \forall y\in B(0,s^{1\over{k^*}}R),
 $$
and
 $$
 ay^{2(k^*-1)}-Ay^{k^*-2}-1\leq Q(x_0,s;y)\leq M_2^2y^{2(k^*-1)}+M_2y^{k^*-2}+1,\ \  \ \forall y\in B(0,s^{1\over{k^*}}R),
 $$
with $M_2=\max_{[0,1]}|m^{(k^*)}(x)|$. In addition, combining Lemma \ref{l3.1} and Lemma \ref{l3.3}, we have
 \bes
 \left.\begin{array}{lll}
 \medskip
 \displaystyle
 W(x_0,s;y)
 &=&
 \displaystyle
 s^{{1\over{2k^*}}}w(s,x_0+s^{-{1\over{k^*}}}y)\\
 &\leq&
 \displaystyle
 Ms^{{{k^*-1}\over{2k^*}}}e^{1-s^{{{k^*-2}\over{2k^*}}}|y|+({3\over a})^{{{1}\over{2(k^*-1)}}}s^{{{k^*-2}\over{2(k^*-1)}}}},\ \ \forall y\in B(0,2s^{1\over{k^*}}R).
 \end{array}
 \right.
 \label{3.3}
 \ees

Let $y_0$ be the maximum point of $W(x_0,s;y)$ on $\overline B(0,2s^{1\over{k^*}}R)$, that is, $y_0$ satisfies
 $$
 W(x_0,s;y_0)=\overline M(x_0,s,R)=\max_{y\in\overline B(0,2s^{1\over{k^*}}R)}W(x_0,s;y).
 $$
From Lemma \ref{l3.3}, it then follows that
 \bes
 W(x_0,s;y)\leq\overline M(x_0,s,R)e^{1-s^{{{k^*-2}\over{2k^*}}}|y|+({3\over{a}})^{1\over{2k^*-2}})s^{{k^*-2}\over{2k^*-2}}},\ \ \forall y\in B(0,s^{1\over{k^*}}R).
 \label{3.3a}
 \ees

We are going to show that $\overline M(x_0,s,R)$ is bounded, independent of all large $s$. For such purpose, we need consider two different cases as follows.

{\it Case 1}.\ \ $y_0\in\partial B(0,2s^{1\over{k^*}}R)$. Thanks to \eqref{3.3} and the fact of ${1\over2}>{{{k^*-2}\over{2(k^*-1)}}}$, one yields
 $$
 \overline M(x_0,s,R)\leq Ms^{{{k^*-1}\over{2k^*}}}e^{1-2Rs^{1\over2}+({3\over a})^{{{1}\over{2(k^*-1)}}}s^{{{k^*-2}\over{2(k^*-1)}}}}\leq M_0
 $$
for some positive constant $M_0$, which is independent of all large $s$ and may vary from place to place below.

{\it Case 2}.\ \  $y_0\in B(0,2s^{1\over{k^*}}R)$. Clearly $\Delta_yW(x_0,s;y_0)\leq0$, and so $Q(x_0,s;y_0)\leq0$. Thus,
 $$
 ay_0^{2(k^*-1)}-Ay_0^{k^*-2}-1\leq0.
 $$
This implies $|y_0|\leq y^*$ for some constant $y^*>0$, independent of all large $s$. Furthermore, for all sufficiently large $s$ and any $y\in B(0,y^*)$,
it is easy to check that the dominant term of $Q$ is given by
 \bes
 \left.\begin{array}{lll}
 \medskip
 \displaystyle
 Q(x_0,s;y)
 &\approx&
 \displaystyle
 s^{-{2\over{k^*}}}\Big\{s^2\Big|m'(x_0+s^{-{1\over{k^*}}}y)\Big|^2+sm''(x_0+s^{-{1\over{k^*}}}y)\Big\}\\
 \medskip
 &\approx&
 \displaystyle
 s^{-{2\over{k^*}}}\cdot s^2(m^{(k^*)}(x_0))^2\cdot y^{2(k^*-1)}s^{-{{2(k^*-1)}\over{k^*}}}
 +s^{-{2\over{k^*}}}\cdot s\cdot m^{(k^*)}(x_0)\cdot y^{k^*-2}s^{-{{k^*-2}\over{k^*}}}\\
 &=&
 \displaystyle
 (m^{(k^*)}(x_0))^2y^{2(k^*-1)}+m^{(k^*)}(x_0)y^{k^*-2}.
 \end{array}
 \right.
 \label{3.4}
 \ees
This shows that $Q(x_0,s;y)$ is bounded, uniformly in all large $s$.

On the other hand, a direct application of the elliptic Harnack inequality (see, for instance, \cite{GT})
to the equation satisfied by $W$ concludes that
 $$
 \max_{y\in\overline B(0,y^*)}W(x_0,s;y)\leq C_0\min_{y\in\overline B(0,y^*)}W(x_0,s;y)
 $$
for some positive constant $C_0=C_0(a,A,M_2)$.  As a result, for all large $s$, we get
 $$
 {{\overline M^2(x_0,s,R)}\over{C_0^2}}\leq\int_{B(0,y^*)}W^2dy\leq\int_{B(0,s^{1\over{k^*}}R)}W^2dy
 =\int_{B(x_0,R)}w^2dy\leq1,
 $$
which implies that $\overline M(x_0,s,R)$ is bounded, uniformly in all large $s$.

In summary, the above analysis shows that $\overline M(x_0,s,R)$ is bounded, uniformly in all large $s$. With loss of generality, we now assume that
$\mu(B(x_0,R))>0$. Then, there is a sequence $\{s_j\}$
and a positive function $W^*$ such that
 $$
 \lim_{j\to\infty}W(x_0,s_j;y)=W^*(y)\ \ \mbox{locally uniformly in}\ \bR,
 $$
and
 \bes
 \left.\begin{array}{lll}
 \medskip
 \displaystyle
 \int_{-\infty}^\infty (W^*)^2dy&=&
 \medskip
 \displaystyle
 \lim_{j\to\infty}\int_{B(0,(s_j)^{1\over{k^*}}R)}W^2(x_0,s_j;y)dy=\lim_{j\to\infty}\int_{B(x_0,R)} w^2(s_j,x)dx\\
 &&\ \ =
 \displaystyle
 \mu(B(x_0,R))=\mu(\{x_0\})>0.
 \end{array}
 \right.
 \nonumber
 \ees
In addition, from the equation satisfied by $W$, $W^*$ solves the following ODE equation
 \bes
 (W^*)''(y)=((m^{(k^*)}(x_0))^2y^{2(k^*-1)}+m^{(k^*)}(x_0)y^{k^*-2})W^*(y),\ \ \ y\in\bR.
 \nonumber
 \ees
This ends the proof of Theorem \ref{th1.3}(1-i). \ \ \ \ \  $\hfill \Box$

\vskip10pt {\it Proof of Theorem \ref{th1.3}\,(1-ii)}\ \ Assume that $x_0=0$ or $x_0=1$. The proof is similar to the case of $x_0\in(0,1)$, and so we only sketch it.
Denote
 $$
 W(x_0,s;y)=(\mu(s,R))^{-{1\over2}}s^{-{1\over{2k^*}}}w(s,x_0+s^{-{1\over{k^*}}}y),
 $$
and
 $$
 Q(x_0,s;y)=s^{-{2\over{k^*}}}\Big\{s^2\Big|m'(x_0+s^{-{1\over{k^*}}}y)\Big|^2+sm''(x_0+s^{-{1\over{k^*}}}y)+c(x_0+s^{-{1\over{k^*}}}y)-\lambda_1^\mathcal{N}(s)\Big\},
 $$
where
 $$
 \mu(s,R)=\int_{B(x_0,R)\cap(0,1)}w^2(s,x)dx,
 $$
and $R>0$ is chosen such that $\mu(B(x_0,R)\cap(0,1))=\mu(\{x_0\})>0$. Note that $\lim\limits_{s\to\infty}\mu(s,R)=\mu(\{x_0\})$.
As in the case of $x_0\in(0,1)$, we have
 $$
 \Delta_yW=Q(x_0,s;y)W,\ \  \ \forall y\in B(0,s^{1\over{k^*}}R)\cap\Omega_s,
 $$
and
 $$
 ay^{2(k^*-1)}-Ay^{k^*-2}-1\leq Q(x_0,s;y)\leq M_2^2y^{2(k^*-1)}+M_2y^{k^*-2}+1,\ \  \ \forall y\in B(0,s^{1\over{k^*}}R)\cap\Omega_s,
 $$
for some positive constants $a$ and $A$, where $\Omega_s=\{y:\ x_0+s^{1\over{k^*}}R\in(0,1)\}$ and $M_2=\max_{[0,1]}|m^{(k^*)}(x)|$.

Thus, the argument of Theorem 2 (ii) of \cite{CL1} can be adapted to conclude that there is a sequence $\{s_j\}$
and a positive function $W_*$ such that
 $$
 \lim_{j\to\infty}W(x_0,s_j;y)=W_*(y)\ \ \mbox{locally uniformly in}\ \bR_*,
 $$
where $\bR_*=(0,\infty)$ if $x_0=0$ and $\bR_*=(-\infty,0)$ if $x_0=1$, and $W_*$ satisfying
$\int_{\bR_*}(W_*)^2dy=1$ solves the following linear ODE problem
 \bes
 (W_*)''(y)=((m^{(k^*)}(x_0))^2y^{2(k^*-1)}+m^{(k^*)}(x_0)y^{k^*-2})W_*(y),\ \ \ y\in\bR_*.
 \nonumber
 \ees
\ \ \ \ \  $\hfill \Box$

\end {document}